\documentclass[12pt,reqno]{amsart}
\usepackage{graphicx}
\usepackage{amsmath}
\usepackage{amssymb}
\usepackage{upgreek}
\usepackage{marvosym}
\usepackage{amsthm}
\usepackage{caption}
\usepackage{subcaption}
\usepackage{amstext}
\usepackage{array}
\usepackage{euler}
\usepackage{times}
\usepackage{lipsum}
\usepackage{kpfonts}
\usepackage{algorithm}
\usepackage{algorithmicx}
\usepackage{algcompatible}
\usepackage{longtable}
\usepackage[margin=1in]{geometry}
\newtheorem{thm}{Theorem}[section]

\newtheorem{prop}[thm]{Proposition}
\theoremstyle{definition}

\theoremstyle{remark}

\numberwithin{equation}{section}

\def\tto{\;{\lower 1pt\hbox{$\rightarrow$}}\kern-10pt
\hbox{\raise 2pt\hbox{$\rightarrow$}}\;}

\usepackage[english]{babel}
\usepackage{natbib}
\newcommand{\bff}{\bf F}
\begin{document}
\title[Robust Toll Pricing]{Robust toll pricing: A novel approach}
\author{T. Dokka}%
\address{T. Dokka: Department of Management Science, Lancaster University, UK}%
\email{t.dokka@lancaster.ac.uk}%
\author{A.B. Zemkoho}
\address{A.B. Zemkoho: School of Mathematics, University of Southampton, Southampton, UK}
\email{a.b.zemkoho@soton.ac.uk}
\author{S. Sen Gupta}
\address{S. Sen Gupta: Department of Economics, Lancaster University, UK}
\email{s.sengupta@lancaster.ac.uk}
\author{F.T. Nobibon}
\address{F.T. Nobibon: Fedex Europe, Belgium}
\email{tallanob@gmail.com}
\thanks{A preliminary version of this work has appeared in Proceedings of ATMOS'2016}
\thanks{The work of the second author is partially supported by the EPSRC grant with reference EP/P022553/1}%
\keywords{Toll-pricing, conditional value at risk, robust optimization}%

\date{\today}%
\begin{abstract}
We study a robust toll pricing problem where toll setters and users have different level of information when taking their decisions. Toll setters do not have full information on the costs of the network and rely on historical information when determining toll rates, whereas users decide on the path to use from origin to destination knowing toll rates and having, in addition, more accurate traffic data. Toll setters often also face constraints on price experimentation which means less opportunity for price revision. Motivated by this we propose a novel robust pricing methodology for fixing prices where we take non-adversarial view of nature different from the existing robust approaches. We show that our non-adversarial robustness results in less conservative pricing decisions compared to traditional adversarial nature setting. We start by first considering a single origin-destination parallel network in this new robust setting and formulate the robust toll pricing problem as a distributionally robust optimization problem, for which we develop an exact algorithm based on a mixed-integer programming formulation and a heuristic based on two-point support distribution. We further extend our formulations to more general networks and show how our algorithms can be adapted for the general networks. Finally, we illustrate the usefulness of our approach by means of numerical experiments both on randomly generated networks and on the data recorded on the road network of the city of Chicago.
\end{abstract}
\maketitle
\tableofcontents
\section{Introduction and Literature} \vspace{0.5pc}

\subsection{Introduction}
Road networks played a crucial role in economic and social development acting as trade enablers. Hence they find an important place in every government's policies. There has been much debate on how roads building should be funded. Traditionally it has been the case that roads were built and maintained by the funds collected from the public in the form of taxes. However, many economists, researchers and policymakers questioned this practice \cite{lindsay_06}. The main critique being that significant proportion of the tax payers may not be using the road being built. In fact, in his book \emph{The Wealth of Nations}, Adam Smith argued ``When the carriages which pass over a highway or a bridge ... pay toll in proportion to their weight Šthey pay for the maintenance of those public works exactly in proportion to the wear and tear which they occasion of them. It seems scarce possible to invent a more equitable way of maintaining such works". This idea has gained much more attention in last few decades and ever more popular today than before. As a result private investment in road building has seen a significant increase. Another main reason is that often not enough tax is collected by governments, especially in developing countries to fund large road building projects. To tackle this governments are encouraging and attracting private players in road building and many of these projects are now done under Public-Private-Partnership (PPP) framework. The PPP-type model is widely adopted due to advantages like bridging the fiscal gap and also efficient project management practices of private sector companies compared to governmental agencies. For example in India, many new highways are built with PPP type of model after the introduction of the amendment of National Highway Act 1956 in the year 1995, which enabled private investors participation in highway construction and maintenance, see \cite{Singh_IJPM06} and references therein.
 Typically these projects employ build-operate-transfer model. Here the investing company enters in a contract with government to build a road/highway. In return of the investment, the company is allowed to collect tolls for an agreed period of time before the transfer of ownership to government. In fact, tolls have become a primary way to encourage private investment in public infrastructure; see \cite{Brown_AusER05}. There are both successes and failures of this model. One of the notable examples is the M6 toll between Cannock and Coleshill, which opened in 2003. According to a BBC News Report, ``the company operating M6 toll made a 1 million pound loss in the year 2012",``drivers have said the road is underused because of its prices". Therefore, a key element to the success of this model is the revenue generated from  tolls. The investor company's main objective is to maximize the revenue from tolls. Hence the ``right" toll price can be the defining factor to the success of the project. The key to a successful revenue maximization pricing mechanism lies in understanding the network users options compared to the toll road.\vspace{0.5pc}


 In \cite{labbe_mgmt_sci}, a bilevel model is proposed to capture the situation where the toll-setter anticipates the network user's reaction
to his decisions. In a full information situation it is assumed that costs of travel on the network are fixed and known to both toll-setter and users. However, cost/time of travel is rarely constant over time in a real world transportation network. Depending on many factors such as weather, day and time of week, accidents, etc., there can be considerable variation in cost. Having said this our ability to have a reasonably good estimate of travel time has never been better with latest technology able to provide us with almost real time traffic updates.This means users may change their decisions over time depending upon then costs/times of travel in the network. Toll-setter, however, suffers from the disadvantage that (if not always, more often in practice) he is not allowed to change the toll very frequently due to policy regulations and other constraints. In most cases, toll is required to be fixed for at least a minimum period. Even if kept unchanged for a minimum period, changing toll price and especially increasing, usually has a negative impact on user's beliefs and may end up resulting in reduced revenues. In such a situation toll-setter has to make his decisions under uncertainty about user's future options. On the other hand users have full (or reasonbly accurate) information before they make their decisions. \vspace{0.5pc}


In this work we study a robust toll-pricing mechanism which aims to minimize the risk of the toll-setter against this uncertainty. In doing so, we use the ideas from robust optimization literature and show that our approach is very near to the conditional value-at-risk approach used in portfolio optimization and other problems.

\subsection{Literature}
Profit and revenue maximization problems over a transportation network are given much attention in pricing literature, see for eg., \cite{vanhoesel_ejor, bouthou_ijoc_2007,kara_focs2004} to name a few. Within a huge body of papers, many have studied the application of bilevel programming paradigm to pricing problems, such as \cite{labbe_mgmt_sci}; and many subsequent papers, \cite{cote_jrpm2003}, \cite{bouthou_networks2007}, \cite{labbe_jrpm2010}, \cite{tuncel_cor2016}, \cite{alain_aor} applied bilevel framework to  several different application areas. A deterministic version of the problem we study in this paper has been investigated in \cite{labbe_mgmt_sci,labbe_trans_sci,labbe_4or,labbe_do,labbe_networks,vanhoesel_ejor,bouthou_ijoc_2007}. However, the stochastic extensions of the problem have gained more interest only in recent years. Two different stochastic extensions of the model in \cite{labbe_mgmt_sci} have been studied in \cite{savard_ts2015} and \cite{savard_partb2013}. In \cite{savard_ts2015}, authors study the logit pricing problem. \cite{savard_partb2013} studies the two-stage stochastic problem with recourse extension of deterministic toll pricing problem also taking view of limited price revision opportunity. In this paper we study the robust toll pricing problem mainly on single commodity parallel networks and show that the approach can be extended to general networks. By single commodity we mean the network has a single origin and single destination and by parallel we mean there can be several roads connecting origin and destination. The deterministic pricing problem on such networks is easy to solve using a closed-form formula. That is given costs on the alternative roads set the toll to the least cost. To the best of our knowledge, there is no work on robust pricing in the presence of uncertainty even in such basic networks, and as we will show that the pricing problem in these networks can itself be quite rich. There, however, are two studies where robust optimization framework is applied to pricing problems, in \cite{violin_thesis}  and \cite{gardner_TS_partc}. In both these works the models considered are different from our model and problem setting. Furthermore, the budgeted uncertainty model considered in \cite{violin_thesis} gives tractable models but may not be best in terms of robustness as found in \cite{Dokka_Goerigk_atmos17}. Understanding the pricing problem in parallel networks will provide useful insights into the complexity of pricing for more general networks involving more commodities and with variable demands. As we will show that the ideas we propose in this work will provide a basis for solving toll-pricing problem in more general networks.
\vspace{0.5pc}

The broader scope of the paper is to propose a bilevel type methodology to pricing problems with limited pricing (or price revising) power. Within this broader scope we are mainly inspired by the toll pricing problem faced by a risk averse toll setter.
Our aim in this work is to provide a better understanding of the toll-pricing problem faced by a risk-averse toll-setter when there is uncertainty on non-toll costs. We use the framework of distributional robustness which is very useful in making optimal decisions under limited or imprecise information, see \cite{goh_OR10} for recent developments on distributionally robust optimization.  Our work also fits into the emerging literature on general static and dynamic pricing that studies pricing problem faced by a seller with insufficient information about demand. \cite{bergemann_sclag_JEEA08} and \cite{bergemann_sclag_JET11} study robust static pricing problems who formulate the problem as minmax regret problem. Within the dynamic pricing literature we mention here \cite{besbes_zeevi_OR09, denboer_zwart_OR15, lim_shanthikumar_OR07, keskin_zeevi_OR14, denboer_zwart_MS14}. Also see the survey on studies on dynamic pricing in \cite{denboer_survey}.  While most of this literature studies dynamic pricing and learning and earning problems, to the best of our knowledge none take into consideration of the case when seller has limited price revising opportunity due to price controls imposed.
\vspace{0.5pc}

\nocite{DZSNb_16}
The rest of the paper is organized as follows: problem definition, assumptions and some notation are described in Section \ref{model}; Section \ref{method} gives the description of the main robust model proposed in this paper; followed by Section \ref{inner_prob} where we discuss the inner or lower level problem in our bilevel formulation; in Section \ref{soln_algs} we give two algorithms: one exact based on the MIP formulation for inner problem given Section \ref{inner_prob} and a two point heuristic for solving toll pricing problem; Section \ref{char_UFN} discusses and illustrates characteristics of our robust model; Section \ref{comp_exps} gives the numerical performance of two point algorithm on simulated and real data sets; Extensions and future work are discussed in
Section \ref{future}.

\section{Problem definition}\label{model}

We will first describe the deterministic pricing model as used in \cite{labbe_mgmt_sci}.
We consider a single-commodity transportation network with a single origin and single destination, $G = (N, A)$, where $N$ (of cardinality $n$) denotes the set of nodes,  and $A$ (of cardinality $m$) the set of Arcs.  The arc set $A$ of the network $G$ is partitioned into two subsets $A_1$ and $A_2$, where $A_2$ denotes the set of roads which are toll-free (\emph{public roads}), and $A_1$ the set of roads which are owned by a toll-setter (\emph{toll roads}). There can be two parallel roads between any two nodes in $G$.

With each toll arc $a$ in $A_1$, we associate a generalized travel cost composed of two parts:  toll ($r_a$) - set by toll-setter expressed in time units, and  non-toll cost ($c_a$) - which in our case varies over time (discretized into unit intervals). An arc $a\in A_2$ only bears the non-toll cost {\bf $c_a$}. Once the toll is set on arcs in $A_1$, it cannot be changed for $T$ consecutive time periods. We will refer to $T$ consecutive time periods in which toll is fixed as \emph{tolling period}. At the end of the tolling period toll setter may be able to revise his price. However, in this paper we only consider static pricing problem which can still be used in dynamic case but does not explicitly optimize pricing decisions over time. We denote $b\in R^n$ the fixed demand, with the assumption that all nodes except origin and destination nodes have a demand equal to 0. Assuming fixed demand and neglecting congestion implies users choose shortest paths between the origin and destination. Further we assume that when faced with two equal alternatives a user will choose the one which maximizes the revenue of toll-setter. Another key assumption is that it allows conversion from time to money and assumes it to be uniform throughout the users. In other words, this can be seen one user using network every time period. Under this setting when the non-toll costs are known to both toll-setter and user the question toll-setter faces is:\emph{How to set prices which maximizes the total toll revenue when the network user chooses shortest path to minimize his cost?}\vspace{0.5pc}



The deterministic problem is well understood both conceptually and algorithmically. Our focus in this paper will be to extend the above deterministic model to the case when there is uncertainty about non-toll costs $c_a$. That is, non-toll costs are not constant and can vary over time. Our uncertainty model and assumptions are as follows:

\begin{itemize}
\item In our model toll-setter has the historical information encoded in the form of previously observed states. A state $s$ corresponds to an observed state of the network in a single time period. In other words, in each state $s$ the non-toll cost on each arc $a\in A$ is fixed denoted as $c^s_a$. The advantage of modeling uncertainty in this way is that correlations between different arcs of network are captured in the states.
\item The number of states equals to $\#H \times T$. That is, toll-setter observes $\#H$ tolling periods. Hereafter we will write $T$ for the rest of all tolling periods assuming toll setter cannot change his price in the future.

\item  The cost distribution (unknown to toll-setter) of each arc is assumed to be fixed and belongs to a set of non-negative distributions $D$ with support in $\Omega = [q,Q]$. One can also consider different supports for different arcs, however, we see $\Omega$ as the aggregated support set.

\item We assume the variability on each arc is bounded, that is the variance-to-mean ratio for the toll period is bounded by a constant which is unknown to the toll-setter. This is usually the case in real world networks.
\end{itemize}

Given this setting our aim is to answer the following question faced by a toll-setter:\vspace{0.5pc}

\emph{How to set toll prices under uncertainty of non-toll costs, when network costs are random with unknown distribution?}

We will now study robust pricing methodologies to answer this question. Starting with a simple two link parallel network we first review more popular robust methodology and then propose a new robust methodology.

\section{Robust Model}\label{method}
Consider a simple parallel network with just two parallel arcs connecting the origin and destination. Let one of these arcs be the toll arc and the other arc is the non-toll arc whose costs are not known. We assume for the ease of exposition that the non-toll costs on the toll arc are zero or negligible. We will later remove this assumption and show the method can be extended to such a case. As mentioned in previous section toll-setter has a sample of costs of $\#H$ tolling periods from the recent history. Using this sample, toll-setter wishes to calculate the toll on the toll arc. Hereafter, we will refer to non-toll arc as $a$. In the rest of the section we will drop the suffix $a$ for the ease of notation and readability. If toll-setter knows the distribution $F$ (we denote the density of $F$ with $\bff$) of $c$ then to fix the toll which maximizes his expected revenue he solves the following optimization problem which maximizes his expected revenue:

\begin{equation}\label{first_eq}
\max_{r\in {\Omega}}  \hspace{0.5pc} \int_{r}^Q r \bff (c) dc.
\end{equation}

To solve this problem note that we can rewrite the objective as $r[1 -  \int_{q}^r \bff(c) dc]$. We can solve the problem by equating the first derivative which is $1-\left( \int_{q}^r \bff(c) dc + r \bff(r)\right)$ to 0. For example if $F$ is uniform distribution with support $\Omega$ then $r$ is obtained by solving $\frac{r-q}{Q-q} + r\frac{1}{Q-q} = 1$, which implies revenue maximizing integer toll is $\frac{Q}{2}$.

 In the absence of this knowledge, a risk -averse toll-setter would prefer to insure his revenues by setting tolls such that the usage of toll arc is maximized as much as possible. On the other hand, setting toll too low, for example close to $q$, will result in high usage but does not necessarily mean better revenue as setting it to a higher price may give much better revenue. Setting it too high may mean no usage and loss of revenues. Suppose that, toll-setter first decides his toll and then nature, who plays adversary to toll-setter, will decide on $F$. Then toll-setter wishes to calculate a robust toll price which maximizes his revenue by solving the following optimization problem.

\begin{eqnarray}
\max_{r\in {\Omega}} & \hspace{0.5pc}  \int_{r}^Q r \bff(c) dc \nonumber\\
s.t. \hspace{0.5pc} \min_{F\in D} & \hspace{0.5pc}  \int_{r}^Q r \bff(c) dc \nonumber\\
s.t. \hspace{0.5pc}  & \underline{u} \leq \mu_F(c)\leq \overline{u} \label{formu1}\\
& \sigma^2_F(c) \leq \overline{\kappa} \mu_F(c) \nonumber
\end{eqnarray}

Here the parameters $\underline{u}$, $\overline{u}$ are calculated as confidence limits of mean; $\overline{\kappa}$ is the belief of toll-setter formed after observing data and also is a parameter controlling the risk averseness of toll setter. This belief can change over time and possibly converge to $\kappa$.  Assuming an adversarial nature is very common in robust optimization and online optimization literature, for example it has been used in \cite{bergemann_sclag_JEEA08} and \cite{lim_shanthikumar_OR07}. We will refer to this as \emph{AN} model. Note that the constraints in (\ref{formu1}) correspond mainly to nature's problem i.e., to find a distribution satisfying mean (or moment) constraint. The second constraint limits the possible distributions by using the assumption of bounded variability.  Such a situation with sufficiently high allowed variability gives too much power to adversarial nature forcing toll-setter (to be too conservative) to  set very low $r$ if he chooses to be robust against all possible $F\in D$. To avoid such over-conservativeness, we propose to consider that nature does not play such a role. Instead nature's objective is to minimize the overall expected cost of the network user, that is:

\begin{equation}
\int_r^Q r \bff(c) dc + \int_q^r c \bff(c) dc,
\end{equation}

where the first term is the expected cost of travel on toll road and second term is expected cost on non-toll road. Toll-setter then solves the following bi-level distributionally robust program to find the robust $r$:

\begin{eqnarray}
  & \max_{r\in {\Omega}} \hspace{0.5pc}  \int_r^Q r \bff(c) dc \nonumber \\
 & \min_{F\in D} \hspace{0.5pc} \int_r^Q r \bff(c) dc + \int_q^r c \bff(c) dc \nonumber\\
s.t. & \underline{u} \leq \mu_F(c)\leq \overline{u} \label{formu12}\\
& \sigma^2_F(c) \leq \overline{\kappa}   \mu_F(c) \nonumber
\end{eqnarray}

We refer to the model (\ref{formu12}) as \emph{UFN} model. To illustrate the difference between nature's role in formulations (\ref{formu1}) and (\ref{formu12}), consider the following example.\\
Suppose $\Omega=[1,100]$ and $r=90$, consider two distributions: $F_1$ puts 0.45 probability mass on 89, 0.5 on 109 and 0.05 on 110; $F_2$ puts 0.135 probability mass on 75 and 0.865 on 104.
 Between these two distributions a nature with objective in (\ref{formu1}) will choose first strategy whereas in (\ref{formu12}) nature will choose the second one as it minimizes expected cost of the user which is 87.975 for $F_2$ as against $89.95$ for $F_1$.

The main motivation behind our model comes from the deterministic bilevel model where, under full information, toll-setter and user have conflicting objectives. However, user's objective is not to make his decisions to decrease toll-setter's revenue but to minimize his own cost. Similarly, in our model the lower level decision maker is acting to minimize the expected cost of user by choosing a distribution which is consistent with the observed mean.\\

Another way of interpreting UFN model is as follows: the distributions available to nature in the lower level problem can be seen as willing-to-pay (WtP) distributions of different users whose mean is consistent with the toll-setter's belief. In the lower level, by choosing a user whose expected WtP is minimum, toll-setter is taking robust decision by choosing a toll which maximizes his revenue from the user with worst expected WtP.\\

We will now relate UFN model with worst-case conditional value-at-risk (CVaR) optimization. To do this we first rewrite the objective function using the assumption of fixed support and then introduce a parameter $\epsilon$ similar to the risk level in CVaR.

Let us start with rewriting the objective function, since we consider $F$ with support in $\Omega$, we can use $\int_r^Q  \bff(c) dc + \int_q^r  \bff(c) dc = 1$ and rewrite the nature's objective function as
\begin{eqnarray}
\int_r^Q r \bff(c) dc + \int_q^r c \bff(c) dc &= r (1 - \int_q^r \bff(c) dc) + \int_q^r c \bff(c) dc  \nonumber\\
& =  r - \int_q^r  (r-c) \bff(c) dc
\end{eqnarray}

 Consider now the following function,
\begin{equation}
f(r, F) =\left[  r - \frac{1}{(1-\epsilon)} \int_q^r  (r-c) \bff(c) dc \right], \nonumber
\end{equation}
where $\epsilon \in (0,1)$. Observe that $f$ is nothing but nature's objective with an additional term involving $\epsilon$. We have the following property of $f$

\begin{prop}\label{prop1}
For a fixed $F\in D$ and $\epsilon \in (0,1)$, $f(r, F)$ is concave and continuously differentiable, and the maximum of $f$ is attained at $r\in \Omega$ such that $\int_{c\leq r} \bff(c) dc = 1- \epsilon$.
\end{prop}
\begin{proof}
Let $G(r) = \int_q^r  (r-c) \bff(c) dc$. From Lemma 1 of  \cite{rockafeller_JR00} $G$ is a convex continuously differentiable function. Using Fundamental theorem of calculus and using differentiation by parts, we can derive $G'(r) = \int_{c\leq r} \bff(c) dc$. This implies $\frac{\partial f}{\partial r} = 1 - \frac{1}{(1-\epsilon)} \int_{c\leq r} \bff(c) dc$, which proves the statement.

\end{proof}

Proposition \ref{prop1} implies that if the distribution of $c$ is known to $F$ and toll setter is interested in finding a toll such that the toll road is used $(100\times \epsilon)$ percent of times in expectation then he should set toll equal to $r$ which satisfies $\int_{c\leq r} \bff(c) dc = 1- \epsilon$. That is, $\epsilon$ should be interpreted as probability that toll road is used at price $r$. In other words, given $F$ when toll-setter decides toll according to (\ref{first_eq}) he indirectly also chooses this probability. This implies that the bi-level problem in (\ref{formu12}) can be written as a parametric single level problem with a max-min objective with parameter $\epsilon$ as follows:
\begin{eqnarray}
&\max_{r\in \Omega;\epsilon\in [0,1]} \min_{F\in D} \hspace{0.5pc}  r  - \frac{1}{1-\epsilon}\int_q^Q \max(r - c,0) \bff(c) dc  \nonumber\\
s.t. & \underline{u} \leq \mu_F(c)\leq \overline{u} \label{formu_new2}\\
& \sigma^2_F(c) \leq \overline{\kappa} \mu_F(c) \nonumber
\end{eqnarray}


For a fixed $F$ the objective function in $f$ is very similar to the concept of Conditional-Value-at-Risk, which has been applied to portfolio optimization problems in \cite{rockafeller_JR00}. In fact, our problem formulation is similar to worst-case conditional value-at-risk studied in \cite{zhu_OR09} and more recently in \cite{toumazis_TS15}.\\

In the rest of the paper we assume time is discretized and we consider the discrete version of (\ref{formu12}) with nature's objective rewritten using fixed support and can be seen as nature optimizing  over samples $C$ drawn from distributions in $D$:

\begin{eqnarray}
\max_{r\in \Omega} & \hspace{0.5pc} r \left[ \frac{1}{T} \sum_{i=1}^T \mathbb{I}_{r\leq c_i} \right] \nonumber\\
 \min_{C\in \Omega^T}   & \hspace{0.5pc}  \left[ r  - \frac{1}{T} \sum_{i=1}^T \max(r - c_i,0) \right] \nonumber\\
s.t.& \underline{u} \leq \mu_F(c)\leq \overline{u} \label{formu3}\\
& \sigma^2_F(c) \leq \overline{\kappa}  \mu_F(c) \nonumber
\end{eqnarray}

A natural way to solve a bilevel problem is to transform it into a single level problem by using optimality conditions and/or using any structure present in the inner problem. Also, the complexity of bilevel problem largely depends on the complexity of the inner or sometimes called follower's or lower-level problem. In next section we study the inner problem of (\ref{formu3}).

\subsection{Inner/Nature's problem}\label{inner_prob}

For a fixed  value of toll price $r$ the inner problem in (\ref{formu3}) is a minimization problem with a concave objective function. Concave minimization problem are hard to solve, for some recent work on minimizing quasi-concave minimization over convex sets see \cite{goyal_orl2013} and references therein. To solve the inner problem in (\ref{formu3}) we reformulate the inner problem as the following non-convex integer programming problem by introducing additional variables:

\begin{eqnarray}
\min_{C\in \Omega^T} &  \left[ r  - \frac{1}{T} \sum_{i=1}^T z_i \right] \label{formu41}\\
s.t.& \underline{u} \leq \mu(c)\leq \overline{u} \label{formu41_1} \\
& \sigma^2(c) \leq \overline{\kappa}  \mu(c)    \label{formu41_2}\\
& c_i - r + z_i \geq 0 \hspace{0.5pc} i=1,\ldots, T   \label{formu41_4}\\
& r - c_i + My_i \geq 0 \hspace{0.5pc} i=1,\ldots, T    \label{formu41_5}\\
& z_i \leq M(1-y_i) \hspace{0.5pc} i=1,\ldots, T     \label{formu41_6}\\
& z_i - (r-c_i)(1-y_i) \leq 0 \hspace{0.5pc} i=1,\ldots, T     \label{formu41_7}\\
& Y\in \{0,1\}; C,Z \geq 0   \label{formu41_8}
\end{eqnarray}

Note that for $M$ in the above formulation any value greater than or equal to $Q$ suffices which gives our next theorem:
\begin{thm}
For a fixed $r$, (\ref{formu41})-(\ref{formu41_8}) is a valid reformulation of the inner problem of (\ref{formu3}).
\end{thm}
\begin{proof}
Constraints (\ref{formu41_4}) - (\ref{formu41_6}) ensure that $y_i=0$ when $r> c_i$ and $y_i=1$ otherwise, and (\ref{formu41_6})- (\ref{formu41_7}) ensure $z_i = \max[r - c_i, 0]$.
\end{proof}
The only non-convex constraint apart from integrality constraints in the above formulation is (\ref{formu41_7}). We linearize this by introducing two additional sets of variables as follows. Replace the product terms $ry_i$ and $c_iy_i$ in this constraint by variables $u_i$ and $v_i$ and then add constraints (\ref{formu4_8})-(\ref{formu4_13}). After doing this we get the following convex integer programming problem.

\begin{eqnarray}
\min_{C\in \Omega^T} &  \left[ r  - \frac{1}{T} \sum_{i=1}^T z_i \right]  \label{formu4}\\
s.t.& \underline{u} \leq \mu(c)\leq \overline{u}  \label{formu4_1}\\
& \sigma^2(c) \leq \overline{\kappa}  \mu(c)   \label{formu4_2}\\
& c_i - r + z_i \geq 0 \hspace{0.5pc} i=1,\ldots, T    \label{formu4_4}\\
& r - c_i + My_i \geq 0 \hspace{0.5pc} i=1,\ldots, T    \label{formu4_5}\\
& z_i \leq M(1-y_i) \hspace{0.5pc} i=1,\ldots, T       \label{formu4_6}\\
& z_i - r + c_i + u_i - v_i \leq 0 \hspace{0.5pc} i=1,\ldots, T     \label{formu4_7}\\
& u_i \leq My_i \hspace{0.5pc} i=1,\ldots, T     \label{formu4_8}\\
& v_i \leq My_i\hspace{0.5pc} i=1,\ldots, T     \label{formu4_9}\\
& v_i \leq c_i \hspace{0.5pc} i=1,\ldots, T      \label{formu4_10}\\
& u_i \leq r \hspace{0.5pc} i=1,\ldots, T     \label{formu4_11}\\
& r - M(1-y_i) \leq u \leq r \hspace{0.5pc} i=1,\ldots, T     \label{formu4_12}\\
& Y\in \{0,1\}; C,Z,U,V \geq 0 \label{formu4_13}
\end{eqnarray}
\begin{thm}
(\ref{formu4})-(\ref{formu4_13}) is a valid reformulation of (\ref{formu41})-(\ref{formu41_8}).
\end{thm}
\begin{proof}
To see this is true note that for every solution to (\ref{formu41})-(\ref{formu41_8}) we can create an equivalent solution to (\ref{formu4})-((\ref{formu4_13})) by taking the $C$, $Z$, $Y$ values as they are and putting $u_i=r$ and $v_i=c_i$ for every $i$ with $y_i=1$ and $0$ otherwise.
\end{proof}
For a fixed $r$, (\ref{formu4}) - (\ref{formu4_13}) can be solved using a state of the art commercial solver like CPLEX and more specialized algorithms are also conceivable owing to tremendous success and availability of techniques for solving convex quadratic integer programs in last few years. We will now look at solving the toll setting problem.

\subsection{Solution Algorithms}\label{soln_algs}

We will first give an exact algorithm  which adds an additional constraint to (\ref{formu4}) - (\ref{formu4_13}) which can be given to a solver like CPLEX; then we move on to simple heuristic.
\subsubsection{Exact Algorithm}
Proposition \ref{prop1} implies that for a fixed $\epsilon\in  \{\frac{1}{T}, \frac{2}{T},\ldots, 1\}$ we can add the constraint (\ref{usage_cons}) to (\ref{formu4}) - (\ref{formu4_13}) and solve the inner problem treating $r$ as variable. This will give us a maximum price to have toll road used $\epsilon T$ times. We will formally give this in Algorithm \ref{inner_formu}.
\begin{equation}\label{usage_cons}
\sum_i \frac{y_i}{T}\leq \epsilon
\end{equation}

\begin{algorithm} [hbtp]
\caption{\label{inner_formu} Robust Toll Algorithm}

\begin{algorithmic}
\STATE{INPUT:  $\underline{u}$, $\overline{u}$, $\overline{\kappa}$}
\STATE{$R=\phi$}
\FOR{$\epsilon \in \{\frac{1}{T}, \frac{2}{T},\ldots, 1\} $}
\STATE   Solve (\ref{formu4}) - (\ref{formu4_13}) + $\frac{\sum_i y_i}{T} \leq \epsilon$
\STATE{Let $r_{\epsilon}$ be solution output; $r_{\epsilon}\in R$}
\ENDFOR
\STATE{OUTPUT: r = $\mbox{arg max}_{{\epsilon}} \epsilon r_{\epsilon}$}
\end{algorithmic}
\end{algorithm}




\subsubsection{Two-point Heuristic}\label{2pt}

The formulation given in (\ref{formu4})-(\ref{formu4_13}) can be hard to solve and can be time consuming when using the generic solvers like CPLEX. Of course, one can derive efficient algorithms using branch and bound and/or other methodologies. In this section, however, we focus on constructing a simple approximate solution to toll pricing problem. In our computational experience of solving (\ref{formu4}) using CPLEX we found that in all cases the solution found has two-point support. That is, the vector of costs returned by CPLEX has exactly two distinct values. If we restrict to the distributions with two-point support $\{\ell,u\}$, assuming $\ell\leq r \leq u$, the nature's problem (\ref{formu4})-(\ref{formu4_13}) can be written as

\begin{align*}
\min_{\ell\in \Omega,u \in \Omega;\lambda\in [0,T]} &  r (T-\lambda) +  \ell\lambda   \nonumber\\
s.t.& (T-\lambda) u + \lambda \ell = \mu T  \nonumber\\
& \underline{u} \leq \mu \leq \overline{u} \nonumber \\
& \lambda(\ell- \mu)^2 + (T-\lambda) (u - \mu)^2 \leq \overline{\kappa} \mu (T-1)
\end{align*}

Suppose that we fix  $\mu = \underline{u}$, we can write the problem of finding of $\{\ell,u\}$ as

\begin{align*}
\min_{\ell\in \Omega,u \in \Omega;\lambda\in [0,T]} &  r (T-\lambda) +  \ell\lambda \nonumber\\
s.t.& (T-\lambda) u + \lambda \ell = \mu T \nonumber \\
& \lambda(\ell- \mu)^2 + (T-\lambda) (u - \mu)^2 \leq \overline{\kappa} \mu (T-1) \nonumber
\end{align*}

Eliminating $u$, we get

\begin{align}
\min_{\ell\in \Omega \lambda\in [0,T]} &   r T -  \lambda (r - \ell)  \label{twopt}\\
s.t. & \lambda(\ell- \mu)^2 + (T-\lambda) ( [\frac{\mu T - \lambda \ell}{(T-\lambda)}]- \mu)^2 \leq \overline{\kappa} \mu (T-1) \nonumber
\end{align}

For a fixed $\lambda$, the objective function in (\ref{twopt}) is linear in $\ell$ with a positive slope. This implies that the solution to (\ref{twopt}) simply is the lowest value satisfying the inequality in (\ref{twopt}) and $u\in \Omega$. Using this observation we now give a simple algorithm for finding a two-point solution to (\ref{formu4}). Hereafter we will denote the integers in $\Omega$ by $\bar{\Omega}$.

\begin{algorithm}[hbtp]

\caption{\label{inner_formu2} Two Point Algorithm}

\begin{algorithmic}
\STATE {INPUT: $\underline{u}$, $\overline{\kappa}$ }
\STATE {Set $BR(r)=0$ for all $r\in \bar{\Omega}\backslash {\underline{u}}$;$BR(\underline{u})=1$}
\FOR{$r\in \bar{\Omega}$ }
\STATE $\lambda = T-1$, $\mu = \underline{u}$
\WHILE{$\lambda \geq 1$}
\STATE $\ell = 0$
\WHILE{$\ell < \mu$}
\STATE $u = \frac{((\mu \times T) - (\lambda \times \ell))}{(T-\lambda)}$
\IF{$u\leq Q$ and $\lambda(\ell- \mu)^2 + (T-\lambda) ( [\frac{\mu T - \lambda \ell}{(T-\lambda)}]- \mu)^2 \leq \overline{\kappa} \mu (T-1)$}
\IF{$obj > (\lambda\times \ell + (T-\lambda)\times r)$}
\STATE $obj = (\lambda\times \ell + (T-\lambda)\times r)$
\STATE {$BR(r) = r\times (T-\lambda)$}
\ENDIF

\STATE break
\ELSE
\STATE $\ell = \ell + 1$
\ENDIF

\ENDWHILE

\STATE $\lambda = \lambda - 1$
\ENDWHILE
\ENDFOR

\STATE {OUTPUT: $\mbox{argmax}_{r\in \bar{\Omega}} BR(r)$}

\end{algorithmic}
\end{algorithm}

Algorithm \ref{inner_formu2} approximately solves by searching for all values of $\lambda$, where $obj$ is the objective in (\ref{twopt}). Note we search for $\ell \in \bar{\Omega}$, this is for numerical simplification and moreover given that $r\in \bar{\Omega}$ this will only result in minor loss in terms of approximation. Note that the heuristic presented in Algorithm \ref{inner_formu2} is aimed mainly for a quick solution where we fix $\mu = \underline{u}$ and $\kappa = \overline{\kappa}$ and the run-time complexity of algorithm is $O(|\bar{\Omega}|^2)$. We now prove that the optimal solution to the nature's problem for a fixed toll is indeed a two-point distribution.

\begin{thm}
The optimal solution to nature's problem is always a two-point distribution.
\end{thm}
\begin{proof}
From \cite{popescu_OR07} it is enough to consider three-point distributions.
In \cite{birge_dula_AOR1991}, it is shown that a two point distribution $(z,r_3)$ with the same mean and variance as the three point distribution can be found, where $r_1 <z< r_2$. Suppose that the optimal three point distribution is $3D = r_1$, $r_2$, and $r_3$ with probabilities $p_1$, $p_2$ and $p_3$. Now consider a two-point distribution $2D = (z,r_3)$ with probabilities $(p_1+p_2,p_3)$ with $r_1 <z< r_2$ with the same mean and variance at most that of 3D. \\

The variance of 2D is at most that of 3D. To see this note that the following inequality is true:

\begin{equation}
\frac{p_1(z-r_1)}{p_2(r_2-z)} \leq \frac{r_2+z}{r_1 + z},
\end{equation}
LHS is equal to 1 due to equivalence of means and RHS is clearly $\geq 1$.
Rearranging terms and adding $p_3r_3^2$ on both sides we get,
\begin{equation}
(p_1+p_2) z^2 + p_3 r_3^2 \leq p_1 r_1^2 + p_2 r_2^2 + p_3 r_3^2.
\end{equation}

To show that $2D$ has a lower expected cost we need to show that
\begin{equation}\label{2dclaim}
(p_1+p_2) g(z) + p_3g(r_3) \leq  p_1g(r_1)+p_2g(r_2)+p_3g(r_3)
\end{equation}
where $g(x) = r$ if $x\geq r$ and $x$ otherwise. \\

Consider first the case when $r=r_2$. From the equivalence of means we have
\begin{equation}
p_1(z-r_1) = p_2(r-z)
\end{equation}
rearranging and adding $p_3r$ on both sides we get
\begin{equation}
(p_1+p_2) z + p_3r = p_1r_1+p_2r_2+p_3r
\end{equation}
which is (\ref{2dclaim}).

The case when $r_2<r$ follows using a similar reasoning as above. Now consider the case when $r_2> r$. It is enough to observe that increasing $r$ to $r_2$ will not change nature's optimal action 3D, in which case setting $r$ to $r_2$ gives improved revenues to toll setter with same nature's response.
\end{proof}

\section{Characteristics of UFN model}\label{char_UFN}
In this section we study the characteristics of UFN model. In Section \ref{ufn_vs_an} we compare the performance of UFN tolls with AN tolls  and illustrate how the optimal solutions of nature's problem react to change in tolls.  Before that we first discuss the model with discrete set of costs.

\subsection{Discrete costs}
Given that we are inspired mainly by toll prices it should be noted that in many cases prices usually are not any continuous values. Instead, prices are almost always numbers with up to two significant digits. This motivates us to consider the case when $\Omega$ is discrete set of prices with minimum price equal to $q$ and maximum equal to $Q$. We will define $\Omega = \{c_j: j\in J \}$, where $J$ is the index set of all allowed prices.  In this case toll pricing problem can be formulated as following:

\begin{eqnarray}
  & \max_{r\in {\Omega}} \hspace{0.5pc}  \sum_{j:c_j\geq r} x_j r \\
 & \min_{X} \sum_{j:c_j\geq r} x_j r + \sum_{j:c_j<r} x_jc_j\\
s.t. & \sum_{j\in J} x_j =1\\
& \sum_{j\in J} c_jx_j = \mu \\
& \sum_{j\in J} x_j c_j^2 - \mu^2 \leq \kappa \mu
\end{eqnarray}

Where $x_j$ is the probability mass assigned to price $c_j\in \Omega$, first constraint corresponds to the mean (first moment) constraint and second constraint to the bounded variability. In this case the nature's problem is simply a linear program, which is easy to solve. Also, using linear programming theory it is easy to note that the optimal solution has at most three support points. However, in numerical experiments we found that in many cases it is either a two point distribution or it is three point distribution with two support points very close to each other. We now use this formulation to illustrate some characteristics of UFN model.

\subsection{AN vs UFN}\label{ufn_vs_an}
In this section we empirically show how UFN model is better and less conservative compared to the usual robust approach of AN model. To compare the behavior of these two models we solved the nature's problem of both models for a discrete $\Omega$ by varying tolls. In AN model, the distribution found is always such that, the lower support point of distribution is very close to the set toll price. Resulting in very low revenues. Under such model toll setter will be forced to set the toll to lowest price for given moment constraints. On the other hand UFN model exhibits a concave like revenue curve. That is, as toll increases revenues may increase up to a point and then decreases. This behavior is illustrated in Figure \ref{fig:usr_fndly_vs_advy} where numerical experiment done has $\Omega = [0,1000]$, with tolls changed between $350$ and $500$, with $\mu = 500$ and $\kappa =60$.
\begin{figure}[ht]
    \centering
   {\includegraphics[width=100mm, height=50mm]{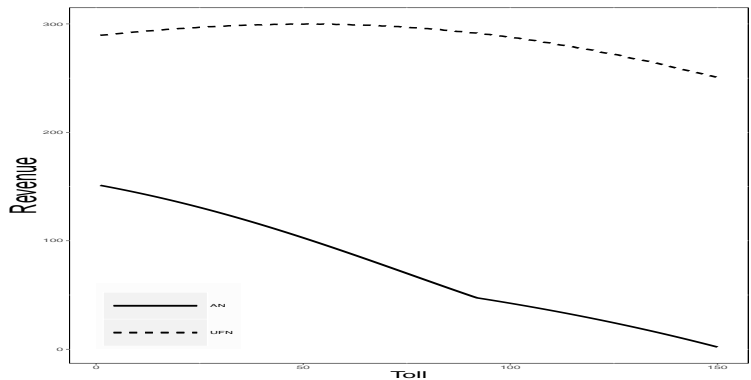} }%
	\caption{Comparison of revenues for varying tolls between user-friendly and adversarial nature}%
    \label{fig:usr_fndly_vs_advy}%
\end{figure}
Figure \ref{fig:UFN} illustrates the changes in the lower support point, revenue and nature's objective values for the UFN two point distribution as the toll is increased with the same parameters as for Figure \ref{fig:usr_fndly_vs_advy}.
\begin{figure}[ht]
    \centering
   {\includegraphics[width=100mm, height=75mm]{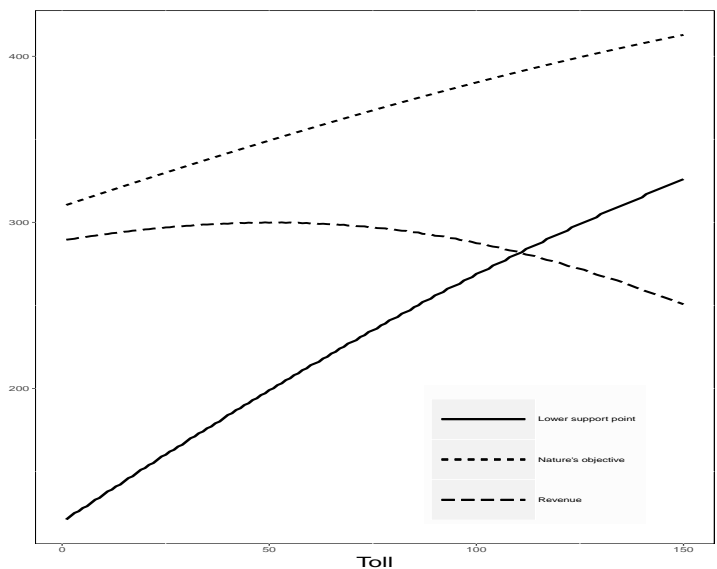} }%
	\caption{Change in two-point distribution as toll increases}%
    \label{fig:UFN}%
\end{figure}

\subsection{Impact of $\kappa$}
As mentioned in Section\ref{method}, $\bar{\kappa}$ indicates the belief of the toll setter about the level of uncertainty. In other words a higher value of $\bar{\kappa}$ compared to $\kappa$ indicates that the toll setter is overly pessimistic. This implies he will try to set a lower toll and hence will have lesser revenue. This is reflected in the illustration given in Figure \ref{fig:toll_vs_var} where we plot revenue from optimal two point UFN distribution as toll and $\bar{\kappa}$  values are changed with mean kept constant. The vertical line of revenues against each toll value correspond to different values of $\bar{\kappa}$ with top values corresponding to smaller $\bar{\kappa}$ . The behavior of the revenue curve is such that the maximizing point (w.r.t the two point distribution found) shifts to the left when $\bar{\kappa}$ value is higher as the toll increases. For the numerical experiment used in Figure \ref{fig:toll_vs_var} the parameters were set to $\Omega = [0,1000]$, tolls changed between $350$ and $500$, $\mu = 500$ and $\bar{\kappa} =[5,60]$. It is also worth noting, with toll closer to mean the variation in the revenue is very less with varying $\bar{\kappa}$.

\begin{figure}[ht]
    \centering
   {\includegraphics[width=100mm, height=75mm]{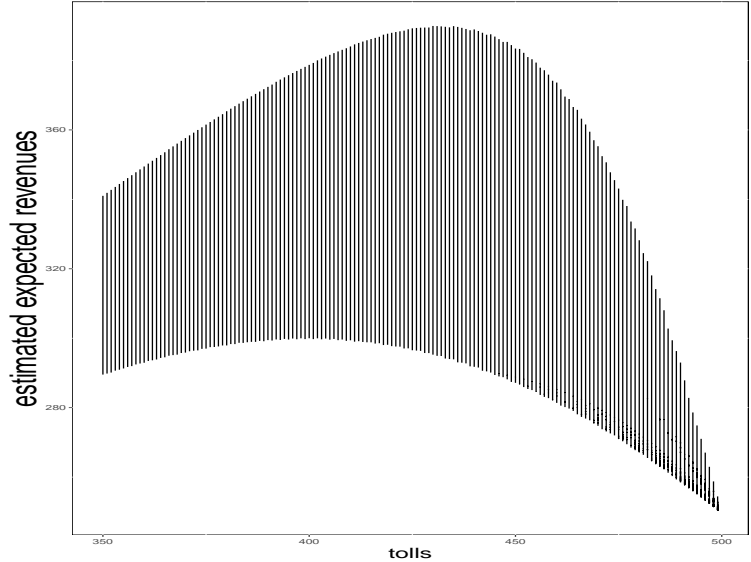} }%
	\caption{Change in two-point distribution as toll increases}%
    \label{fig:toll_vs_var}%
\end{figure}

\subsection{Normal Distribution - Entropy}

In order to assess the performance of UFN model we compared the distribution calculated in the lower level to the case when the distribution of $c$ is Normal with same mean and variance. Figure \ref{fig:2pt_vs_normal} gives the change in cumulative probabilities, for a fixed toll with same mean but changing variance, of two-point distribution calculated in nature's problem versus the actual Normal distribution.The mean is set to $500$ with variance changing between $10$ and $60$ with support in $[0,1000]$. The tolls are set at 300, 350 and 400 for Figures (\ref{fig:300})-(\ref{fig:400}) respectively. The plots give cumulative probabilities in both cases, that is cumulative probability in both distributions at toll value. More formally, plots give $P_{UFN} (c \leq r)$ and $P_{\mathbb{N}}(c\leq r)$ as variance is changed for three different values of r.  It can be seen that two point distributions calculated are very close to the normal case. Furthermore, the distribution gets better when the toll value is increased from 300 to 350, and again the distance between curves increases with increase in toll away from 350.  This indicates the UFN approach will give close to optimal tolls when the distribution of $c$ is Normal given that toll setter has a correct belief about the $\bar{\kappa}$.
\begin{figure}[htbp]
    \centering
	\begin{subfigure}[b]{0.45\textwidth}
		\centering
   \includegraphics[width=\textwidth]{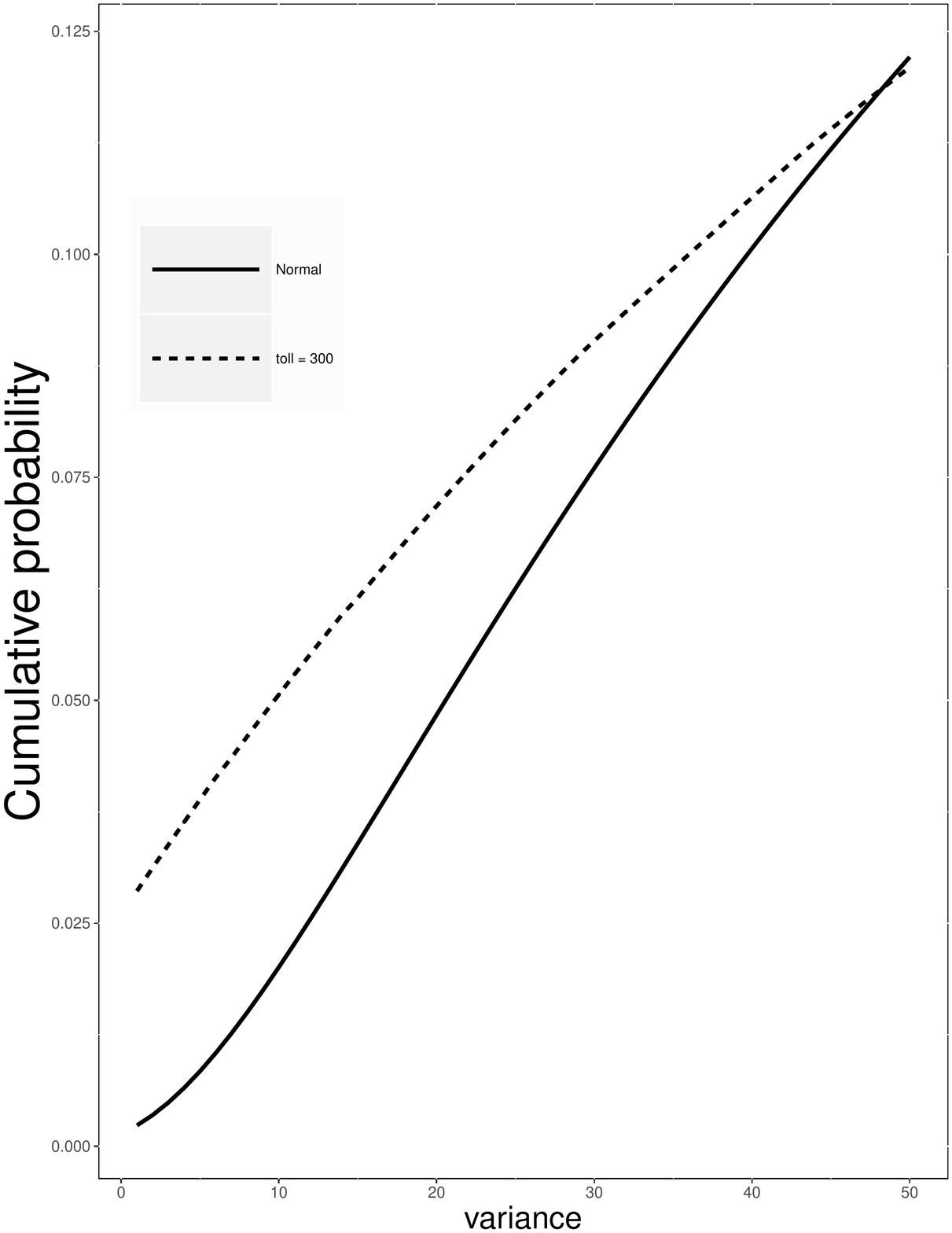} %
	\caption{toll=300}%
	\label{fig:300}
	\end{subfigure}
	~
	\begin{subfigure}[b]{0.45\textwidth}
		\centering
   \includegraphics[width=\textwidth]{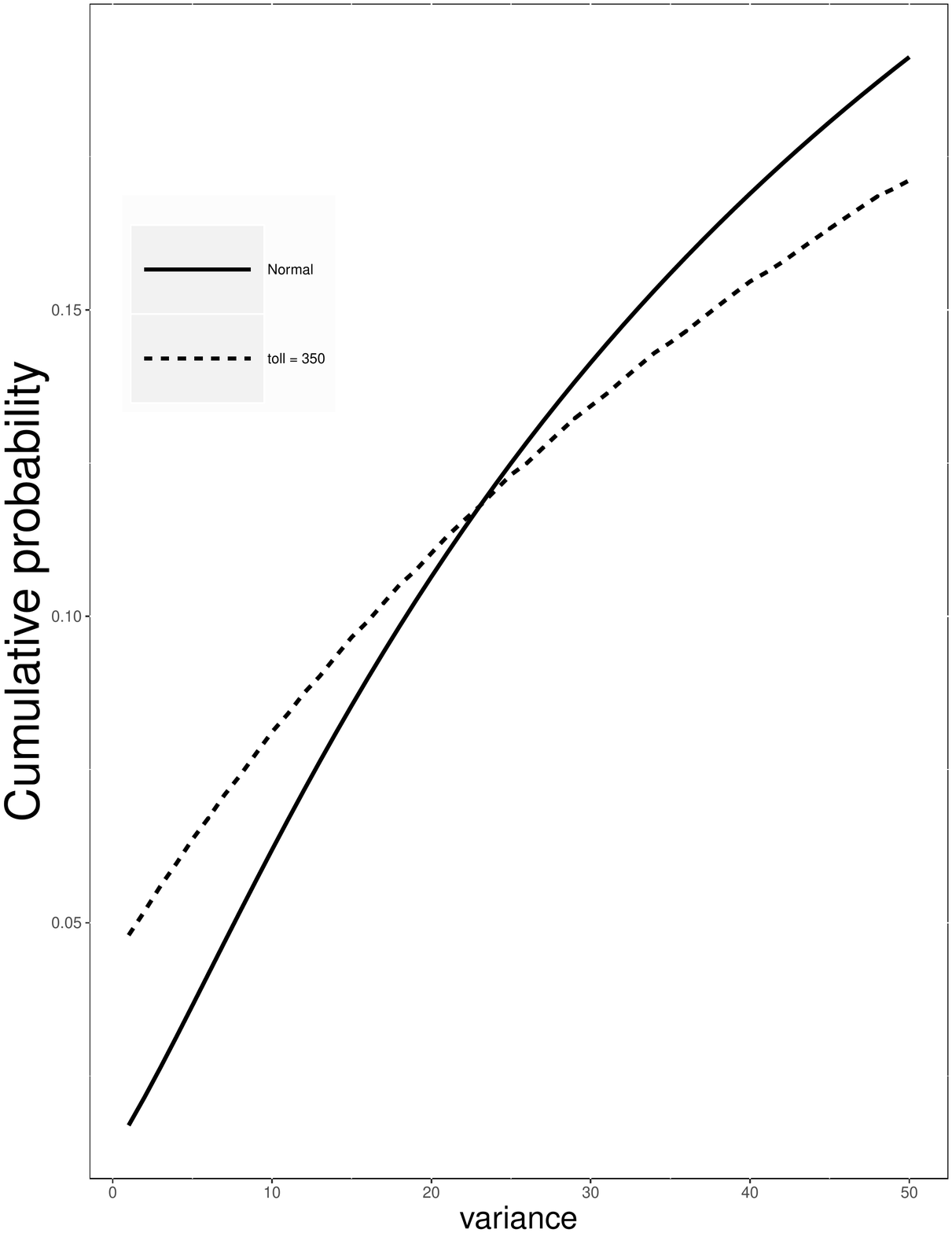} %
	\caption{toll=350}%
	\label{fig:350}
	\end{subfigure}
	~
	\begin{subfigure}[b]{0.45\textwidth}
		\centering
   \includegraphics[width=\textwidth]{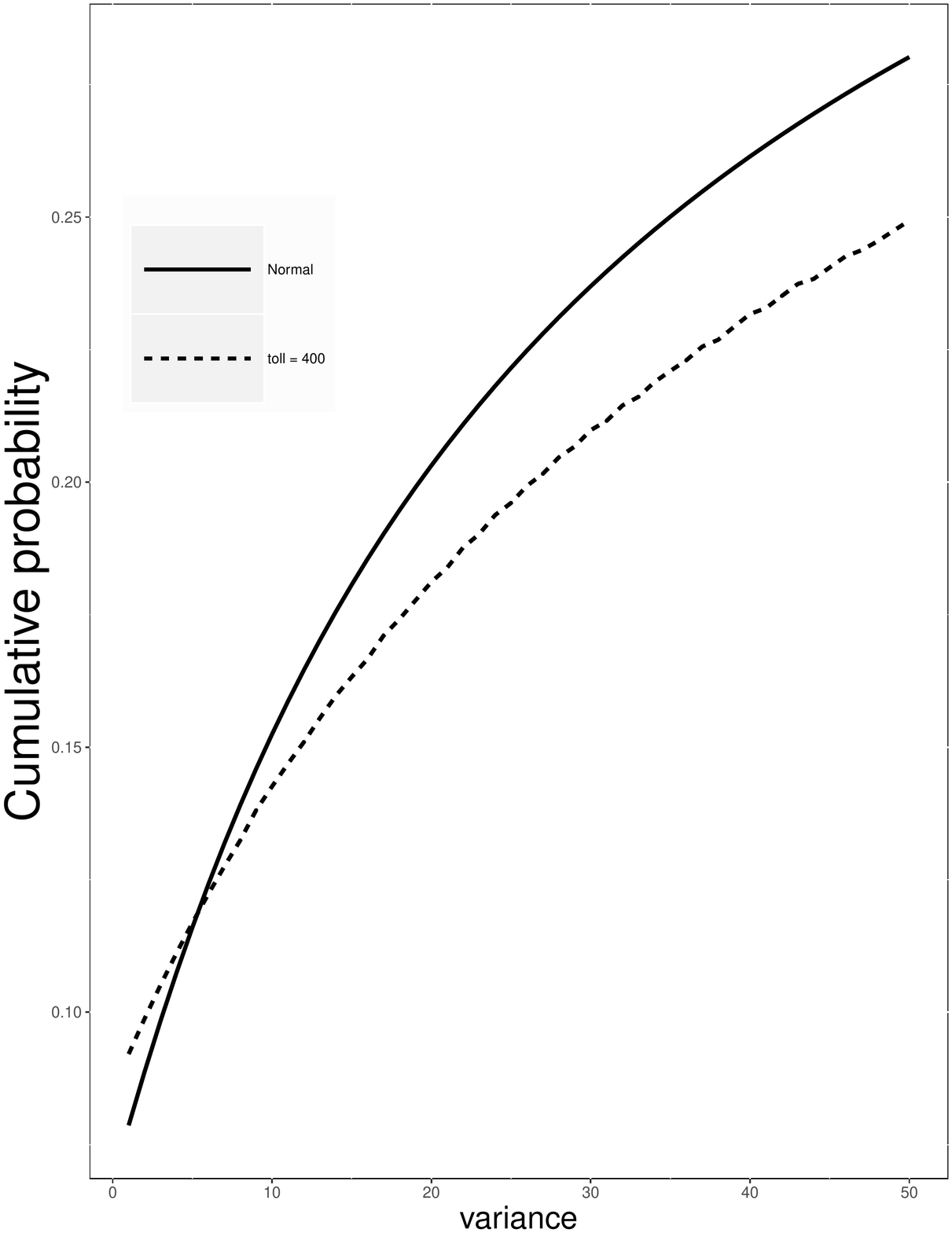} %
	\caption{toll=400}%
	\label{fig:400}
	\end{subfigure}
	\caption{Comparison of two-point distribution calculated at lower level for fixed toll with Normal distribution with same mean and variance}%
    \label{fig:2pt_vs_normal}%
\end{figure}

We believe this behavior is related to maximum entropy probability distributions. The entropy of a probability distribution represents the amount of uncertainty associated with the distribution. The distribution maximizing the entropy is believed to be a good prior distribution. As observed by \cite{roels_OR2008} entropy maximization is not a decision making criteria which can be used under uncertainty, it can only be used as selection criteria for selecting a probability distribution within a stochastic model. Also, it is similar to a barrier function (\cite{boyd_vandenberghe}). The entropy of a discrete random variable with a given distribution is defined as,

\begin{equation}
H(c) = - \sum_{i} p(c_i) \mbox{log} \hspace{0.1pc} p(c_i).
\end{equation}
Our approach can be related to maximizing entropy by replacing the $log p(x)$ with $g(x)$. In other words maximizing (minimizing negative of) a similar measure as entropy. Let $G_r(x)$ be defined as
\begin{equation}
G_r(c) =  \frac{1}{D}\sum_{i} p(c_i) g(c_i),
\end{equation}
where is $D$ is some fixed constant, then $G$ is an approximation of $H$ for each value of $r$. Our approach can be interpreted as approximating $H$ using $G$ by choosing $r=c_i$ where $c_i(1-\sum_{c_i>=r} p(i))$ is maximized. We wonder if it is possible to quantify this relationship and leave it for the future study.

\section{Pricing in General networks}
\subsection{Multiple parallel arcs}

 Let us first consider the immediate extension to a network where there are $k$ non-toll arcs parallel to the toll arc between origin and destination. Let $a_1,\ldots, a_k$ be the non-toll arcs. We input the mean and variance limits ($\overline{u}$, $\underline{u}$ and $\overline{\kappa}$) of the data obtained from taking the following minima, $\min_{i=1}^k c^s_{a_i}$ to calculate the two point robust toll.\\

 Let us now consider the above network when the assumption that the non-toll costs on toll arc are not zero. Let $a_{k+1}$ be the toll arc. To apply our method to this case we calculate the mean and variance limits ($\overline{u}$, $\underline{u}$ and $\overline{\kappa}$) of
\begin{equation}\label{extension}
 \min_{i=1}^k [c^s_{a_i} - c^s_{a_{k+1}}]
\end{equation}
for all $s$.

\subsection{General networks}
 Consider now a general single commodity network with multiple toll arcs, for example the network on the left given in Figure \ref{conversion}. To apply UFN approach we construct an equivalent parallel network as shown in the right side in Figure \ref{conversion}. For each path in this parallel network with toll arcs we will calculate the state minima given in (\ref{extension}) by ignoring all other paths with toll arcs. That is we calculate the robust toll on each toll path as if that is the only path with toll arcs in the network. This we treat as an upper bound on the total toll on that path. We then solve an integer/linear programming problem to allocate the tolls to individual toll arcs. Suppose the upper bounds for the paths in the example network in Figure \ref{conversion} are $\varsigma_1$, $\varsigma_2$, $\varsigma_3$ respectively from left to right, then we solve the following optimization problem to solve for prices of $r_1$, $r_2$, and $r_3$

\begin{align*}
 \max \hspace{0.5pc} r_1+r_2+r_3 \\
s.t \hspace{0.5pc}  r_2+r_3 \leq \varsigma_1\\
 r_1+r_2 \leq \varsigma_2\\
r_1 \leq \varsigma_3\\
 r_i \in \mathbb{Z}  \nonumber
\end{align*}

\begin{figure}[ht]
    \centering
  {\includegraphics[scale=0.4]{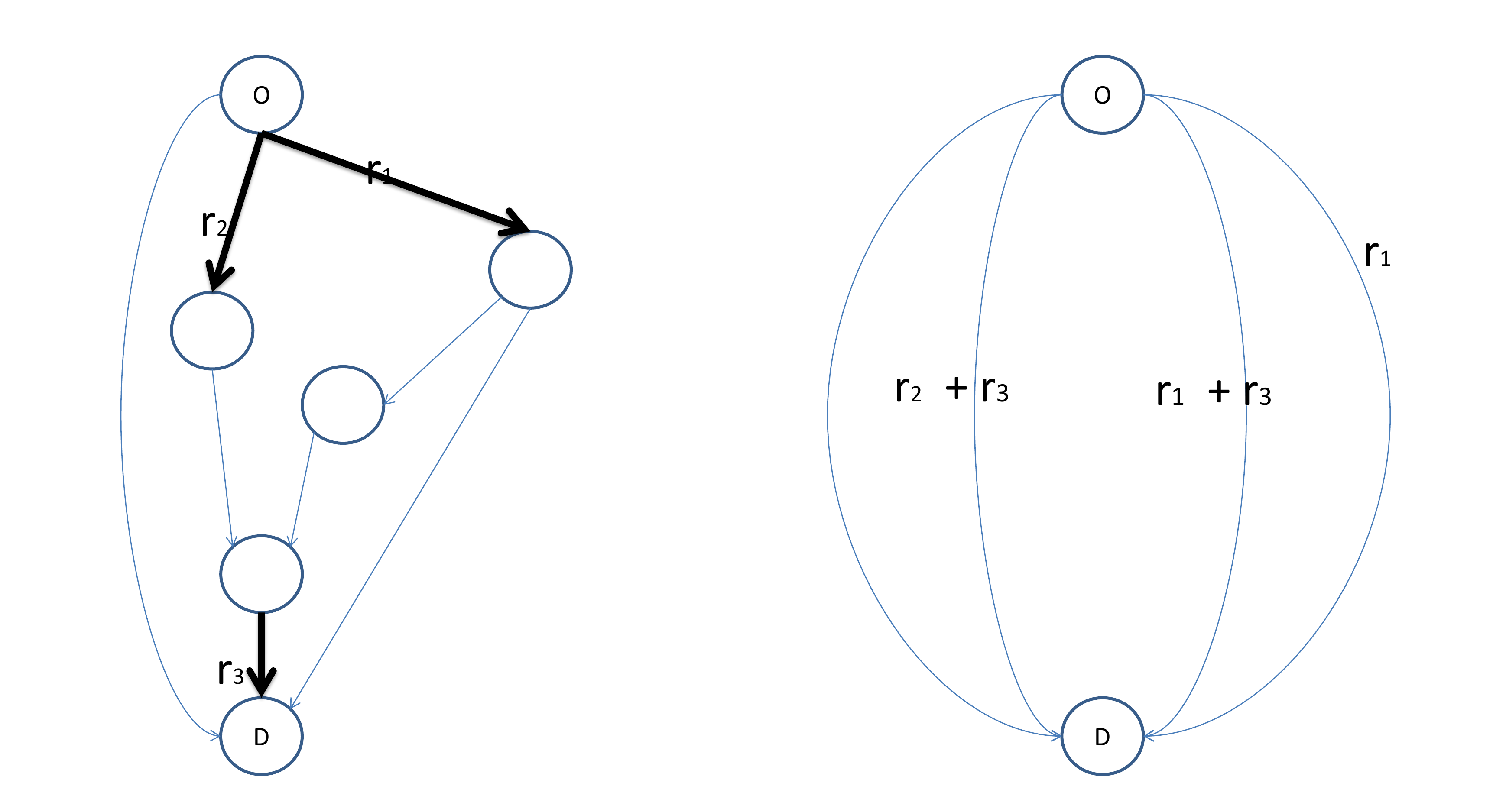} }%
	\caption{General network and an equivalent parallel network}%
    \label{conversion}%
\end{figure}

In theory there can be many paths between any two nodes in a network. However, in practice especially in real world transportation networks the number of paths between any two nodes/cities is usually small or limited. Also, many paths between pairs of cities are far from optimal in any of the scenarios, that is they are usually dominated by a few number of paths. We observed this in Chicago data where there may be several paths between two cities but most often it was only few paths which were optimal with many paths never optimal, not even in a single state. So, in practice even if there are many paths we can still use the above approach.

\section{Computational experiments}\label{comp_exps}

We performed two sets of numerical experiments to assess the UFN model, first set of experiments is on simulated data while the second set of experiments are performed on real-world traffic data collected by City of Chicago. In Section \ref{sim_data} we present our numerical results on simulated data. In Section \ref{chicago_data} we assess the performance of two point toll algorithm on City of Chicago data.

\subsection{Performance metric}
A commonly used metric to measure the robustness of an algorithm in pricing literature is (relative) regret.  We follow this practice and  we use the percentage relative regret from using the robust toll which is calculated as follows:

\begin{equation*}
 \mbox{relative regret}(\%) = \frac{\mbox{optimal revenue  - robust toll revenue}}{\mbox{optimal revenue}},
\end{equation*}
where optimal revenue is revenue obtained by using the optimal price.
From these experiments we want to understand the answers to the following two questions:

\begin{itemize}
\item how bad are revenues from robust tolls compared to the optimal revenues?
\item how do robust tolls compare to optimal tolls?
\end{itemize}
In both experiments we used the two-point approximate algorithm to compute the robust tolls, and we set $T=50$, $\#H = 1$, and $\bar{kappa} = 1$. 

\subsection{Simulated data}\label{sim_data}
In this section we report the performance of our approach with numerical experiments on simulated data. We have done experiments to assess the robustness of our procedure under two different experimental set-ups differing in network structure and cost distributions. We explain them below.

\begin{itemize}
\item {\bf First-Experiment}: We consider a parallel network with five parallel links connecting origin to destination. In this set-up we fix the distributions of the links to be same but allow the parameters to vary randomly within a given interval.

\item {\bf Second-Experiment}: We consider the same network as in first but the distributions on each links can be different including parameters.

\end{itemize}

In both experiments we would like to understand the robustness of the two-point toll. The distributions we use are Beta, Gamma, Normal and Lognormal. The parameters for each distribution are selected uniformly from an interval. The parameter intervals are given in Table \ref{paras}.

\begin{table}
\begin{center}
\begin{tabular}{c c c}
\hline
Distribution	&	First parameter	&	second parameter 	\\
\hline
Beta & [2,5] & [2,5] \\
Beta & [1,3] & [1,3] \\
Gamma & [1,3] & [$\frac13,\frac15$] \\
Normal & [90,110] & [10,30] \\
Lognormal & [0.1,0.3] & [0.1,0.3]\\
\hline
\end{tabular}
\caption{Distributions and parameters used}
\label{paras}
\end{center}
\end{table}

We first created 50 samples (history sample), from each distribution which are used to calculate the robust tolls. We then created 2500 random samples from each distribution and computed optimal revenue generating tolls for each of these samples. We compare the revenues from optimal tolls in each of these 2500 samples with revenues when a robust toll is used which is calculated from a sample in history sample. To calculate an optimal toll for a given instance we try each integer in $\bar{\Omega}$ and select the toll which generates the most revenue. In total we compare robust tolls with optimal tolls on 125000 samples.

\subsubsection{Results and Discussion: Fixed distributions}
In this section we will evaluate the robustness of our two-point robust toll on the instances when all parallel arcs have same distributions but the parameters can be selected randomly in the intervals given in Table \ref{paras}. Table \ref{table:exp2} displays the average percentage relative regret for each of the four distributions. From Table \ref{table:exp2} we observe that the robust toll achieves a regret less than $14\%$ in all distributions with except Gamma all other figures below $10\%$. On the other hand we observed fixing the toll equal to sample mean can have give regret as high as $23\%$. Furthermore, as seen from standard deviation values variation in regrets is also not large. This suggests that revenues from robust toll compare well especially given the fact that the toll decision is taken with minimal knowledge about the network cost distributions.

As previously pointed out a measure of robustness of the toll is how it compares with the optimal revenue generating tolls. Figure \ref{fig:exp21_1}-\ref{fig:exp21_4} displays the comparison of minimum, maximum and average values of tolls over the 50 history samples, and optimal revenue generating tolls in each of the 2500 samples (sorted in increasing order). From Figure \ref{fig:exp21_1}-\ref{fig:exp21_4} we observe that it is possible that the UFN tolls can be too high or too low especially as seen in Normal distributions. However, the average robust toll compares well with optimal tolls and is very close to average optimal toll in almost all with slightly below average in case of Beta. Also the variability in robust tolls displayed in Table \ref{table:exp21} illustrates the robustness of UFN tolls. Figures also suggest that with a higher $\#H$ the variability can be further reduced.

In Figures (\ref{fig:gamma_dyna})- (\ref{fig:b_dyna}), we plot the performance for the four distributions when using average of the robust tolls found using the 50 history samples against the best static optimal toll treating each instance in the set of 2500 as a tolling period but toll setter cannot revise his price. In other words  Here we report the cumulative regrets over time. By cumulative regret we mean regret up to period $t$ calculated against a static optimal toll price. Since we compare against the static optimal price in the earlier periods average robust toll may perform well showing zero regret. As can be seen the cumulative regret from average robust price is less than $2\%$ in all cases.

\begin{table}[ht]
\begin{center}
\begin{tabular}{c c c}
\hline
Distribution	&	Average	& Stdev\\
\hline
\hline
Beta	&	7.62 & 5.77 \\
Gamma	&	13.57 & 12.51 \\
Lognormal	&	8.31	& 8.18 \\
Normal	&	7.36	& 8.12\\
\hline
\end{tabular}
\caption{Fixed case: average (\%) relative regret}
\label{table:exp2}
\end{center}
\end{table}

\begin{table}
\begin{center}
\begin{tabular}{c c }
\hline
Distribution	&	Robust toll		\\

	&	 Stdev \\
\hline
\hline
Beta	&	4.94 \\
Gamma	&	8.56\\
Lognormal	&	5.16	 \\
Normal	&	5.13	 \\
\hline
\end{tabular}
\caption{robust toll variation}
\label{table:exp21}
\end{center}
\end{table}

\begin{figure}[htbp]
    \centering
		\begin{subfigure}[b]{0.45\textwidth}
		\centering
   \includegraphics[width=\textwidth]{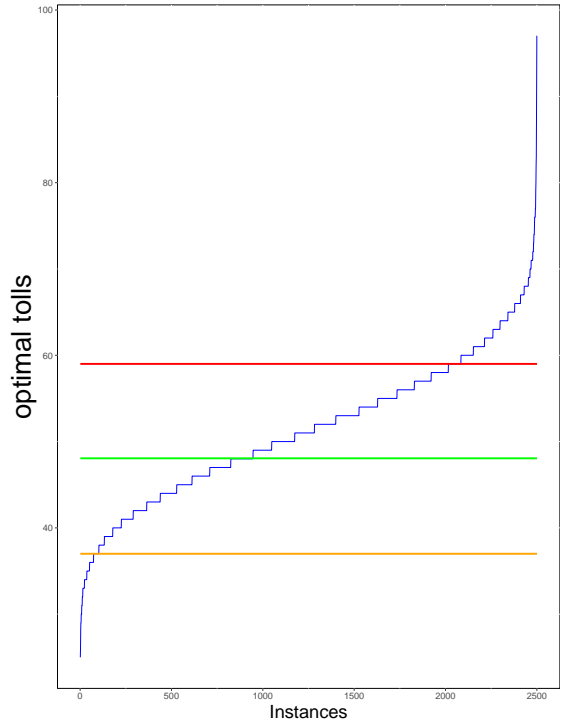} %
	\caption{Beta distribution}%
	\label{fig:exp21_1}
	\end{subfigure}
	~
	\begin{subfigure}[b]{0.45\textwidth}
		\centering
   \includegraphics[width=\textwidth]{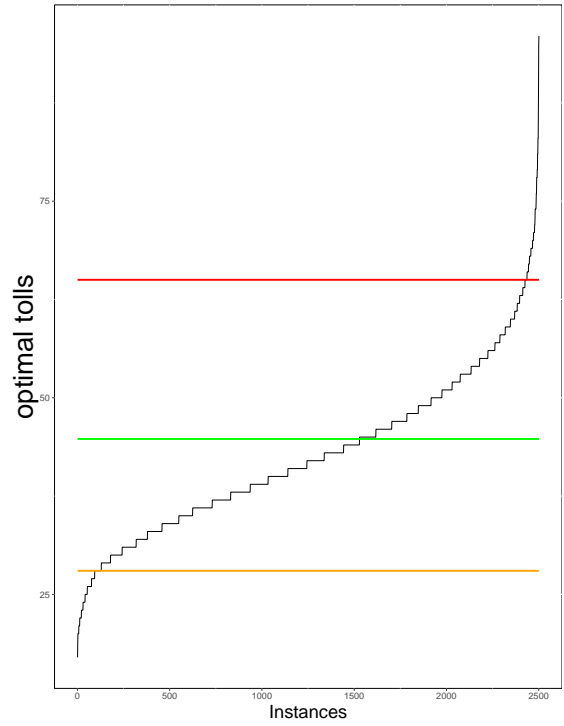} %
	\caption{Gamma distribution}%
	\label{fig:exp21_2}
	\end{subfigure}
	\\
	\begin{subfigure}[b]{0.45\textwidth}
		\centering
   \includegraphics[width=\textwidth]{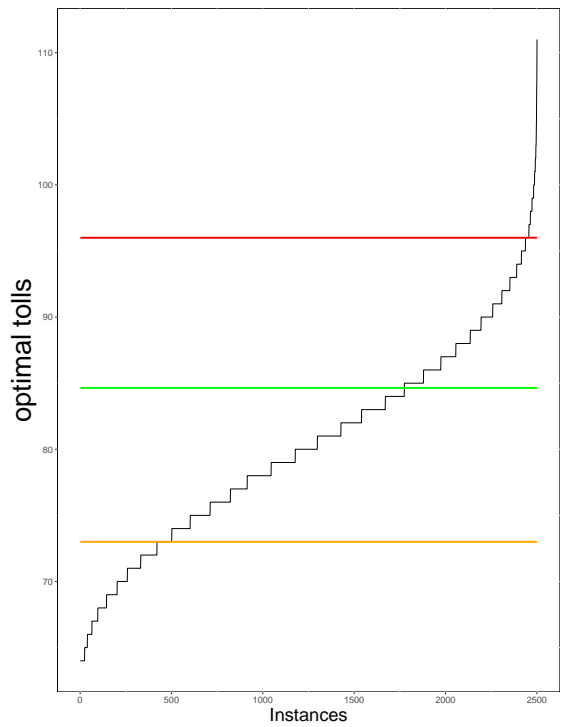} %
	\caption{Lognormal distribution}%
	\label{fig:exp21_3}
	\end{subfigure}
	~
	\begin{subfigure}[b]{0.45\textwidth}
		\centering
   \includegraphics[width=\textwidth]{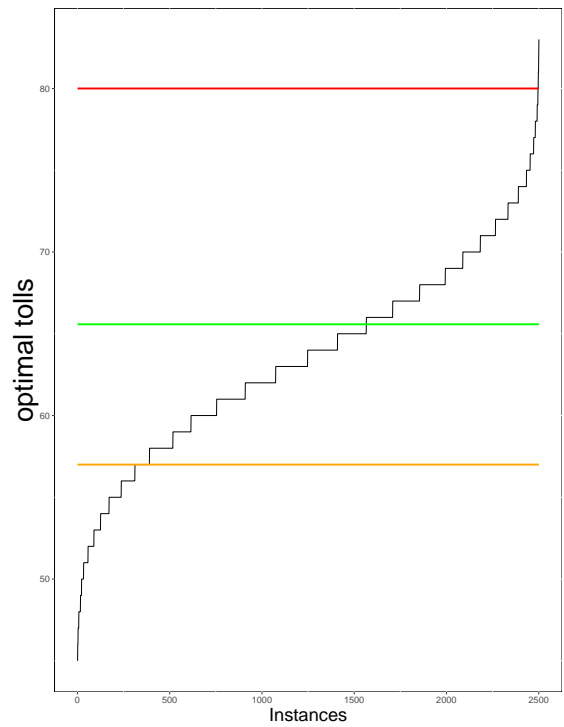} %
	\caption{Normal distribution}%
	\label{fig:exp21_4}
	\end{subfigure}%
		\caption{Comparison of optimal tolls with robust tolls with top, middle and bottom horizontal lines corresponding to maximum, average and minimum robust tolls.}
\end{figure}

\begin{figure}[htbp]
    \centering
	\begin{subfigure}[b]{0.45\textwidth}
		\centering
   \includegraphics[width=\textwidth]{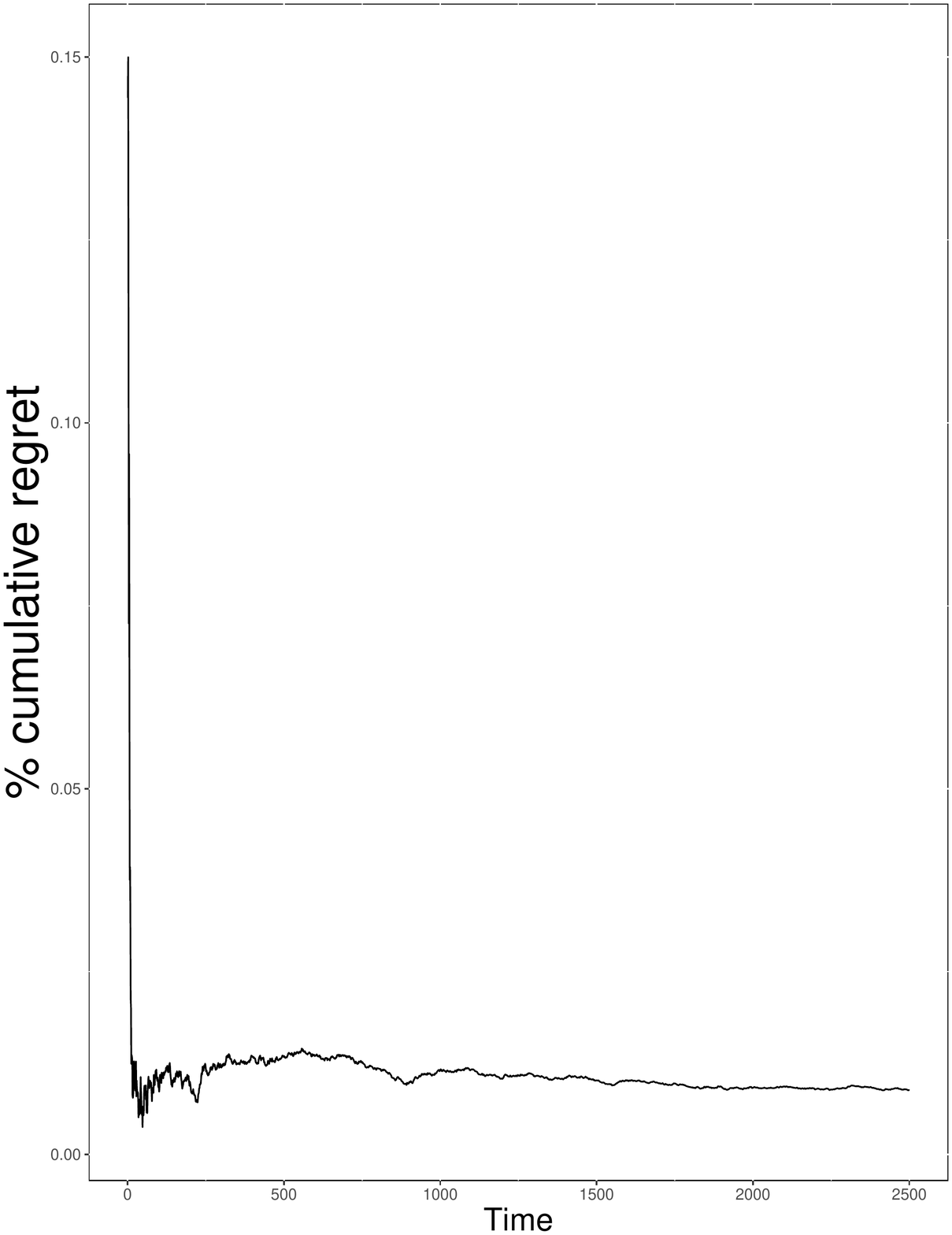} %
	\caption{Gamma distribution}%
	\label{fig:gamma_dyna}
	\end{subfigure}
	~
	\begin{subfigure}[b]{0.45\textwidth}
		\centering
   \includegraphics[width=\textwidth]{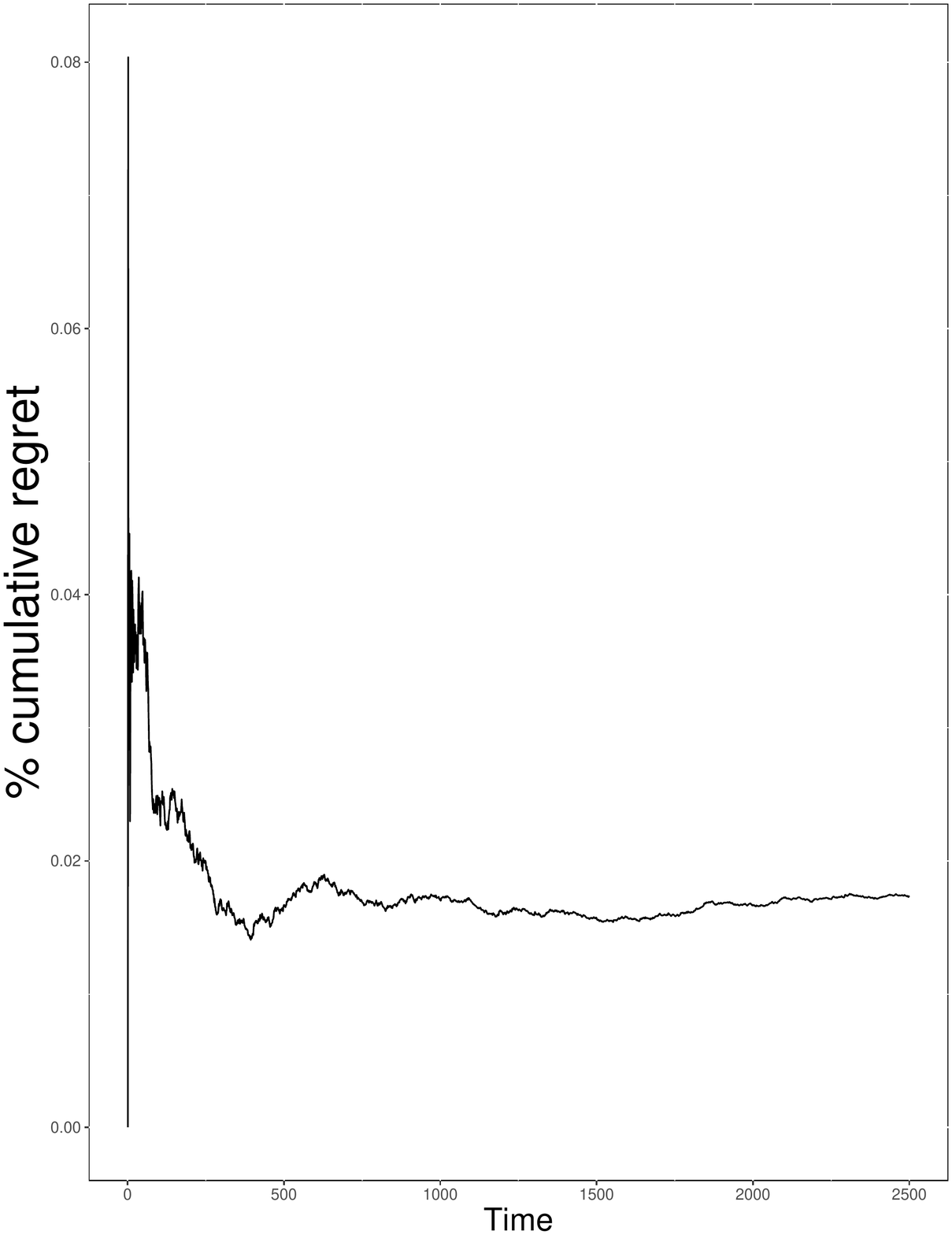} %
	\caption{Lognormal distribution}%
	\label{fig:l_dyna}
	\end{subfigure}
	~
	\begin{subfigure}[b]{0.45\textwidth}
		\centering
   \includegraphics[width=\textwidth]{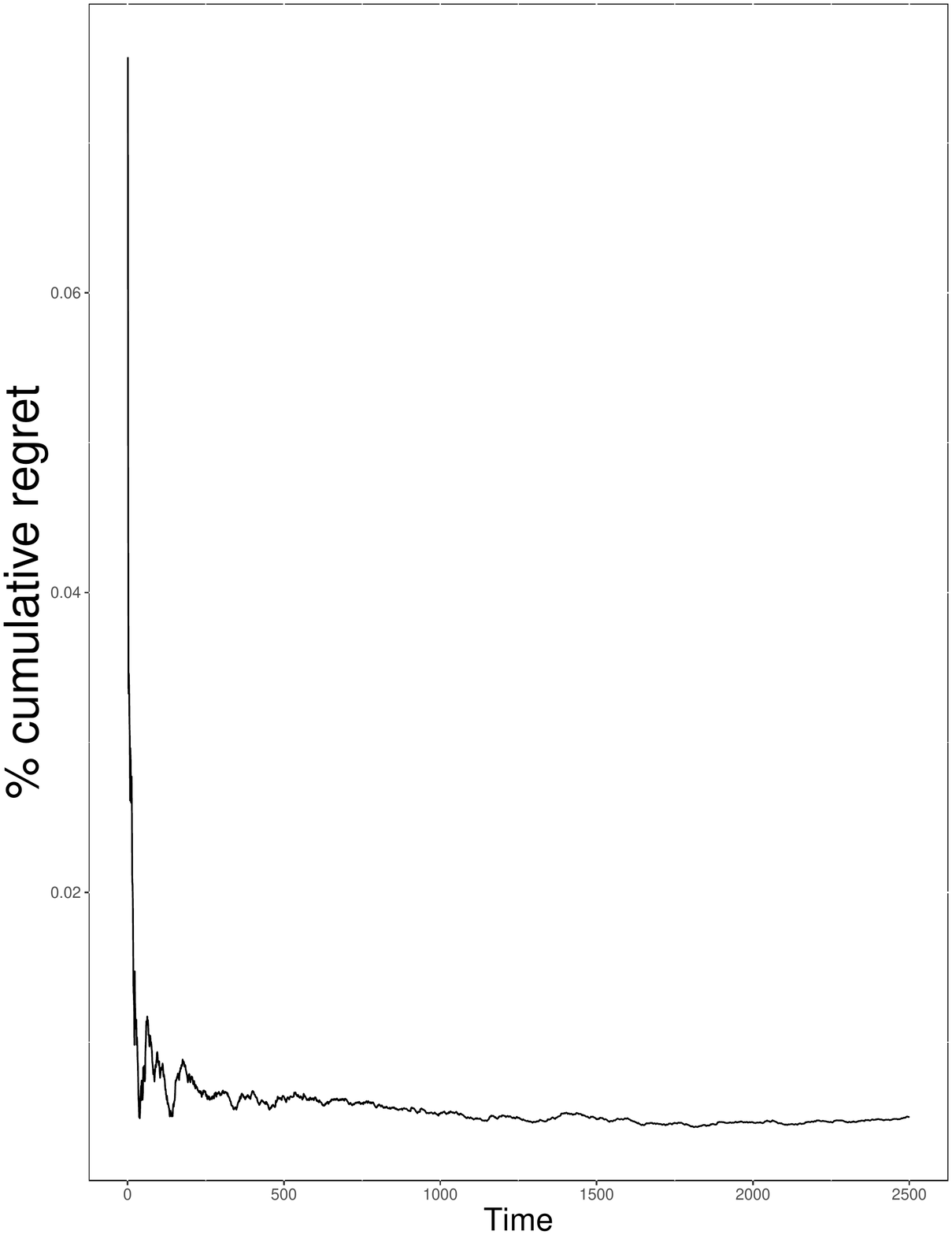} %
	\caption{Normal distribution}%
	\label{fig:n_dyna}
	\end{subfigure}
	~
	\begin{subfigure}[b]{0.45\textwidth}
		\centering
   \includegraphics[width=\textwidth]{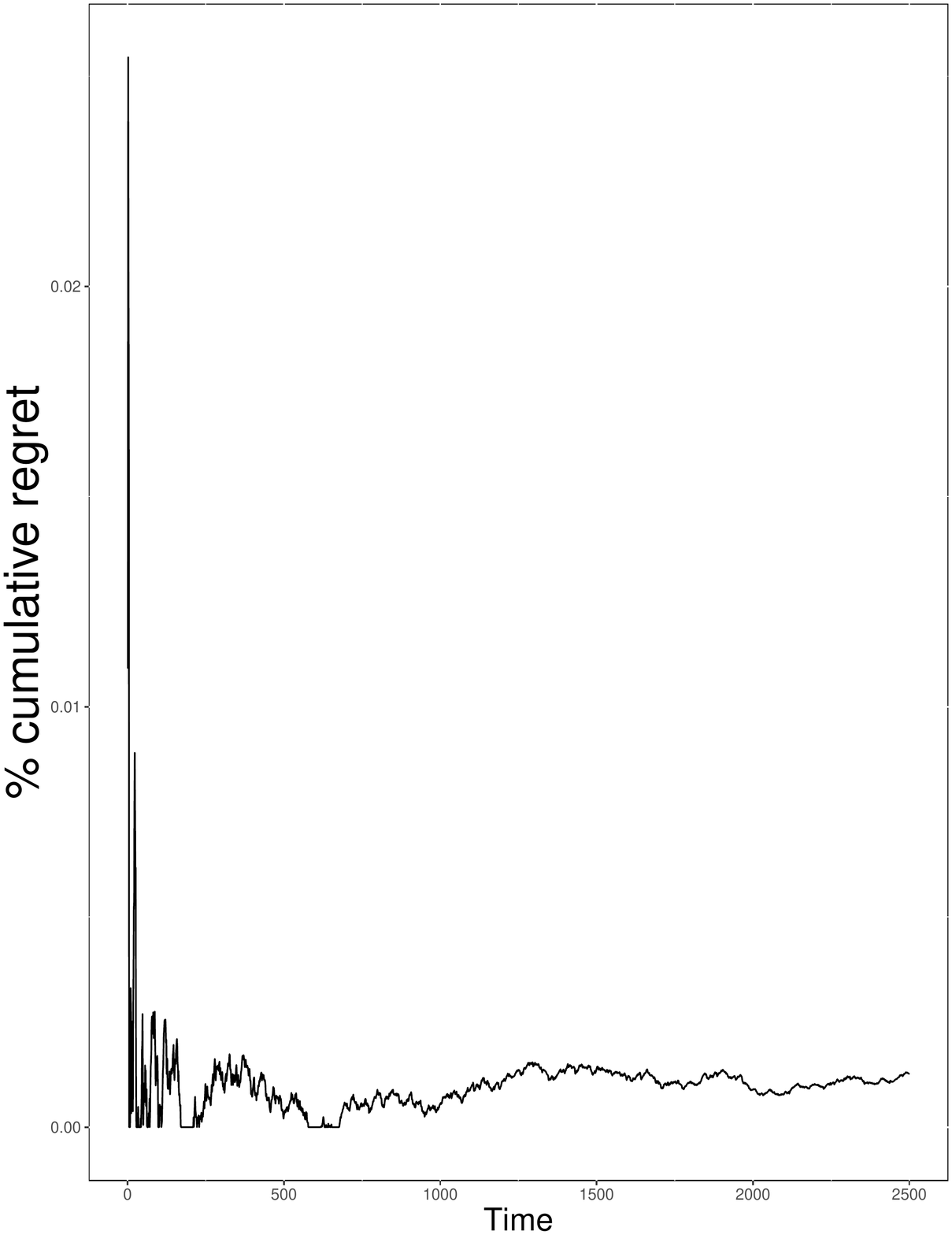} %
	\caption{Beta distribution}%
	\label{fig:b_dyna}
	\end{subfigure}
		\caption{dynamic case performance.}
\end{figure}

\begin{table}[ht]
\begin{center}
\begin{tabular}{c c c}
\hline
Distribution	&	Average	& Stdev\\
\hline
\hline
Beta	&	6.44 & 5.77 \\
Gamma	&	10.2 & 12.51 \\
Lognormal	&	6.73	& 8.18 \\
Normal	&	5.11	& 8.12\\
\hline
\end{tabular}
\caption{Fixed case: average (\%) relative regret with average robust toll}
\label{table:ave1}
\end{center}
\end{table}
\vspace{1pc}

\subsubsection{Results and Discussion: Mixed distributions}

In this section we will evaluate the robustness of UFN toll when arcs in the (same) network can have different distribution with parameters again chosen randomly from intervals given in Table \ref{paras}. We observe from Table \ref{exp3} that average regret from the robust toll is less than that in the case of fixed distribution case. This is also reflected in Figure \ref{fig:exp3} which again displays the comparison of minimum, maximum and average values of robust tolls with optimum revenue generating tolls. The average robust toll is again slightly higher than but very close to the optimal average.

\begin{table}
\begin{center}
\begin{tabular}{c c c c c}
\hline
Distribution	&	Average	 & Stdev \\
\hline
\hline
Mixed	&	7.84\%	& 6.2\%  \\
\hline
\end{tabular}
\caption{Exp 3: average (\%) relative regret}
\label{exp3}
\end{center}
\end{table}

\begin{figure}[htbp]
    \centering
  {\includegraphics[scale=0.3]{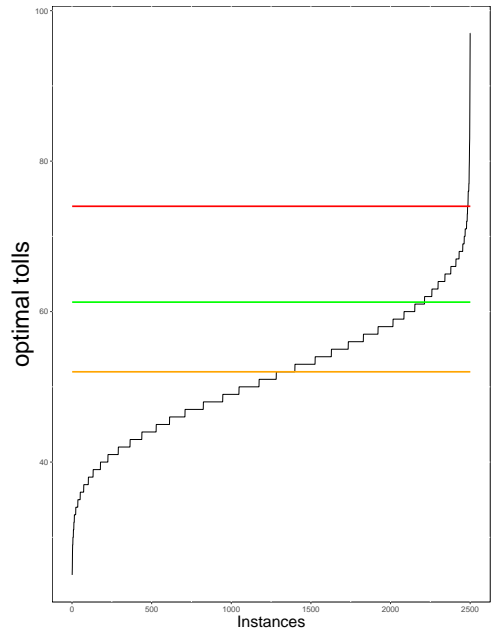} }%
	\caption{Exp. 3: Comparison of optimal tolls with robust tolls}%
    \label{fig:exp3}%
\end{figure}

\subsection{Real data}\label{chicago_data}
In real world the times of travel and hence costs may not have any known distribution which can be analytically expressed and may have seasonality and trends over time. Moreover, model which performs well on simulated data may not perform well on real data as modeling assumptions may not be satisfied. To test our approach on real data which as we explain below is given as travel times on a real road network, we assume the generalized cost is proportional to the time taken. We first explain the details of data collection and then give then details of experimental setup and results.

\subsection{Data Collection and Cleaning}


We used data provided by the City of Chicago\footnote{https://data.cityofchicago.org}, which provides a live traffic data interface. We recorded traffic updates in a 15-minute interval over a time horizon of 24 hours for several days between March 28th 2017  to  May 13th 2017. A total of 4363 data observations were used.

Every observation contains the traffic speed for a subset of a total of 1,257 segments. For each segment the geographical position is available, see the resulting plot in Figure~\ref{chicago1} with a zoom-in for the city center. 
There were 1,045 segments where the data was recorded at least once of the 4363 time points. 
We used linear interpolation to fill the missing records keeping in mind that data was collected over time. Segment lengths were given through longitude and latitude coordinates, and approximated using the Euclidean distance.

As segments are purely geographical objects without structure, we needed to create a graph for our experiments. To this end, segments were split when they crossed or nearly crossed, and start- and end-points that were sufficiently close to each other were identified as the same node. The resulting graph is shown in Figure~\ref{chicago2}; note that this process slightly simplified the network, but kept its structure intact. The final graph contains 538 nodes and 1308 arcs.

\begin{figure}[htbp]
\centering
\begin{subfigure}[c]{0.8\textwidth}
        \includegraphics[width=\textwidth]{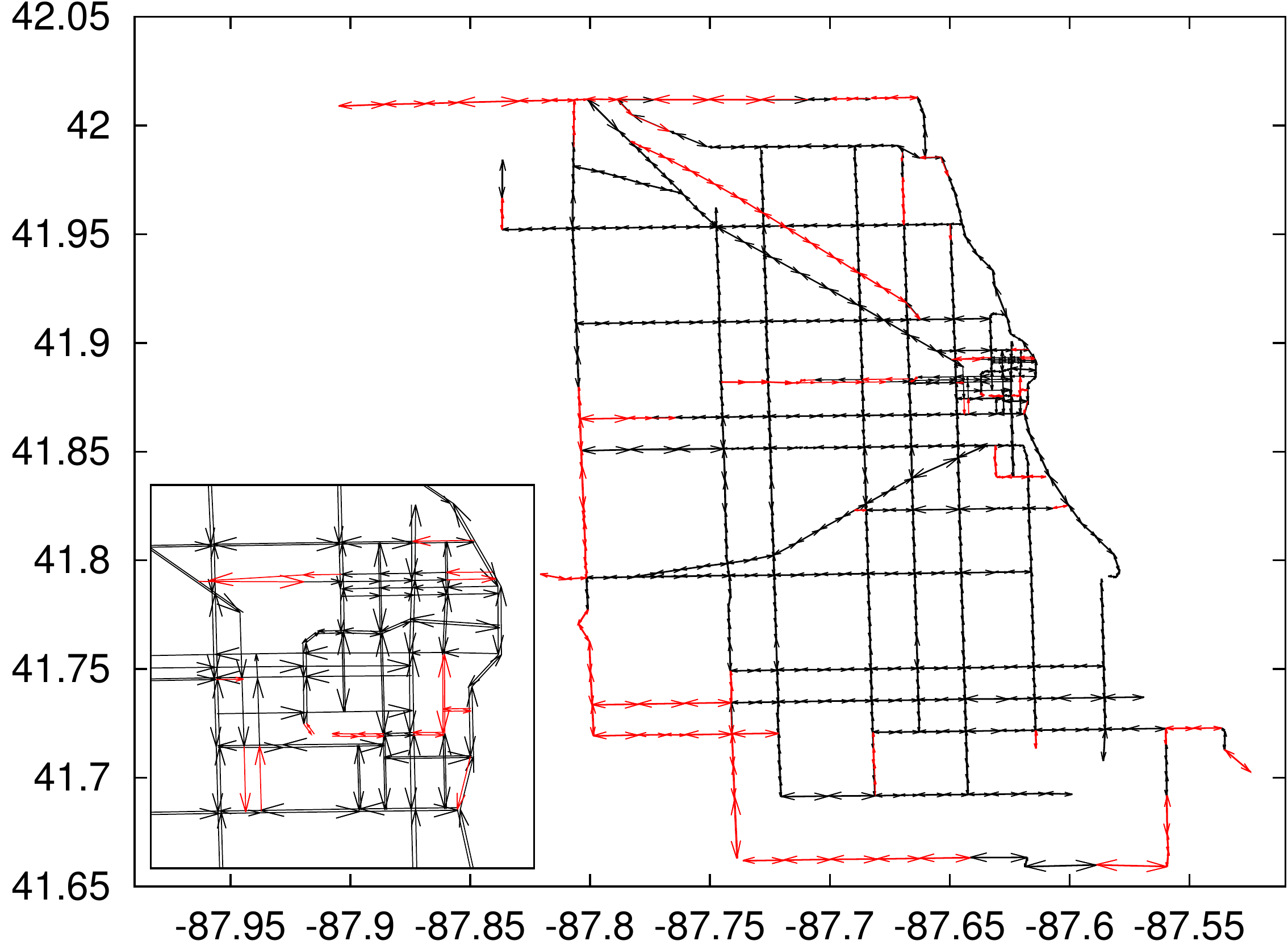}
        \caption{Raw segments with zoom-in for the city center. In red are segments without data.}
        \label{chicago1}
    \end{subfigure}
		~
    \hfill
\begin{subfigure}[c]{0.8\textwidth}
        \includegraphics[width=\textwidth]{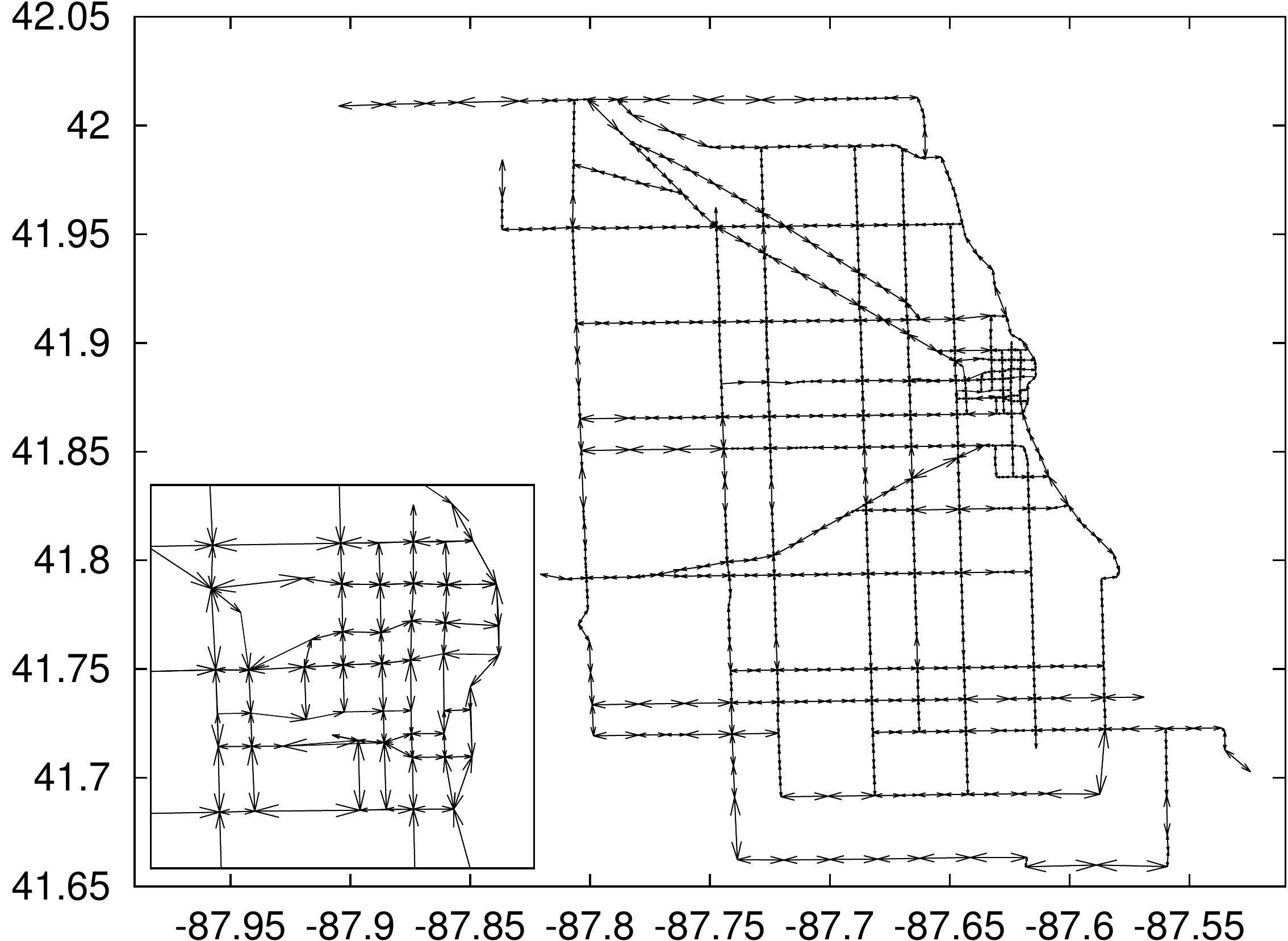}
        \caption{Resulting graph model with zoom-in for the city center.}
        \label{chicago2}
    \end{subfigure}
\caption{Chicago instance}
\end{figure}
\subsection{Experimental Setup, Results and Discussion}
We assume cost of travel on each road in this network is proportional to the time of travel. Since there are no toll roads in this network which we can use for our experiments we have randomly selected 200 pairs of cities and calculated tolls by imagining a toll road between these each of these pair of cities, under the assumption that non-toll cost is zero or such costs have been adjusted in the costs of the other roads. Of 4363 observations we used 2100 as history and all observations for calculating the regrets in all of 200 cases.

The average percentage regret over 200 pairs is given in Table \ref{tab:real}. As can be seen the average regret is less than that observed in simulated case. Interestingly, unlike the case of simulate data UFN tolls were set higher than optimal for all pairs which indicates the non-conservative of the approach. Figure \ref{fig:real_opt_ufn} illustrates distribution of ratio of UFN tolls over optimal tolls. UFN toll was set roughly $12\%$ higher in extreme case with most tolls within $5\%$ of optimal tolls.
\begin{table}
\begin{center}
\begin{tabular}{c c c c}
\hline
	Robust toll	&	mean-variance toll & Robust toll	&	mean-variance toll	\\

	Average	&	Average & Stdev & Stdev	\\
\hline
\hline
	5.64\%	&	30.08\%	& 5.01\% & 5.64\% \\
\hline
\end{tabular}
\caption{real data: average (\%) relative regret}
\label{tab:real}
\end{center}
\end{table}

\begin{figure}[htbp]
    \centering
  {\includegraphics[width=100mm,height=75mm]{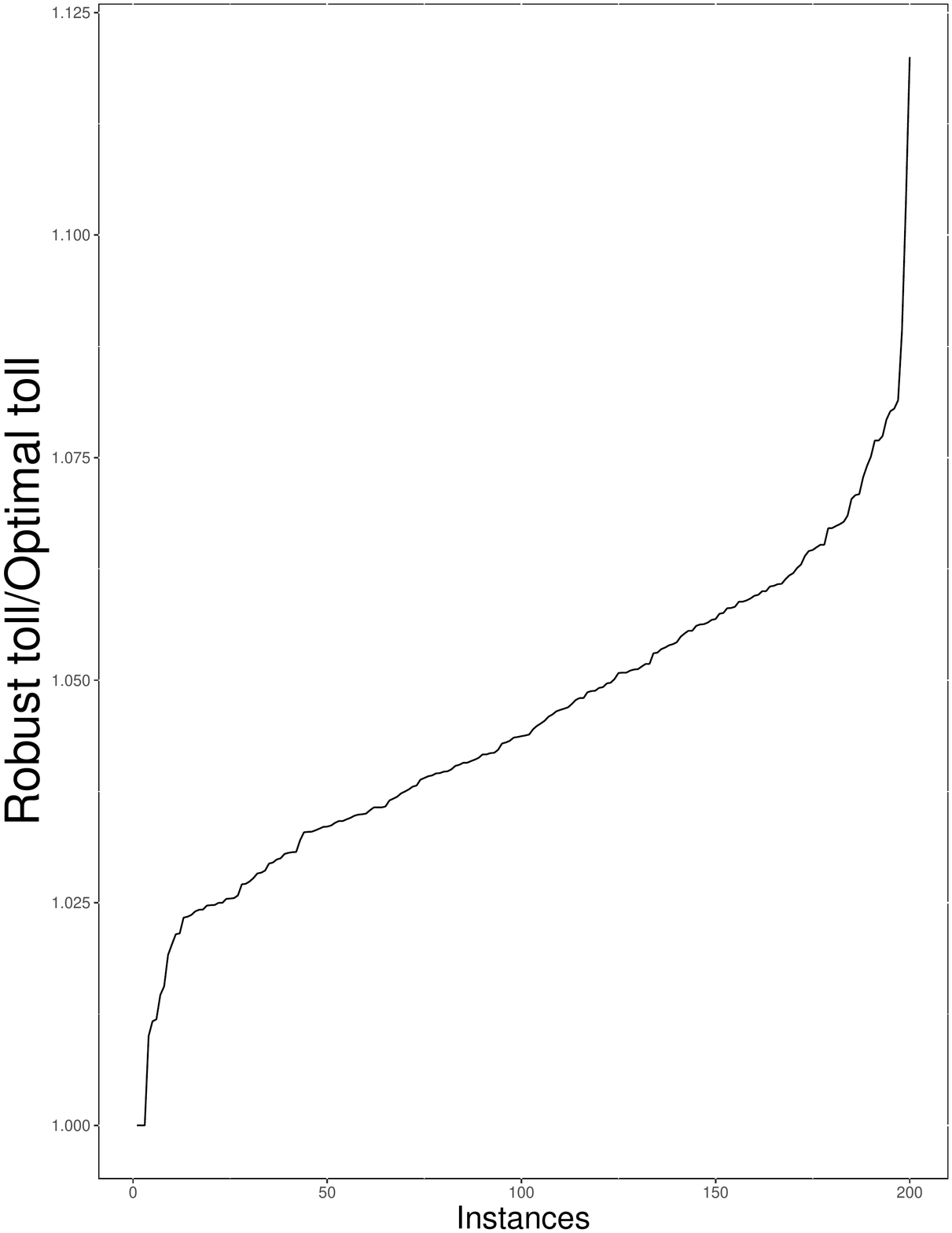} }%
	\caption{Exp. 3: Comparison of optimal tolls with robust tolls}%
    \label{fig:real_opt_ufn}%
\end{figure}

\section{Dynamic pricing and Future Work}\label{future}
In practice a toll setter may be able to revise the toll. However, revising power of toll setter is hugely limited as often they are subject to price controls such as caps on price increase and the number of increases. Also, the assumption of toll setter able to observe the cost may not be true or may be subject to inaccurate assumptions. In \cite{dokka_jacko}, the UFN model is extended to the dynamic case with price controls where toll setter learns the distribution by changing the price dynamically and observing the usage. The new robustness model can also be extended to more general networks such as multi-commodity networks and also including variable demands. We consider these extensions for our future work.

\section*{Acknowledgements}
We would like to thank Martine Labb\'e for an interesting discussion on the topic and also for pointing out references. We thank Marc Goerigk for spotting a mistake in an early version of the paper and for helping with Chicago data.

\bibliographystyle{informs2014trsc}
\bibliography{DZSN_refs}
\end{document}